\documentclass[a4paper,11pt]{amsart}

\usepackage[dvips]{graphics}
\usepackage[matrix,arrow,tips,curve]{xy}
\usepackage[english]{babel}
\usepackage{amsmath}
\usepackage{amssymb}

\usepackage{mathrsfs}

\usepackage{enumerate}

\newtheorem{thm}{Theorem}[section]
\newtheorem{lemma}[thm]{Lemma}
\newtheorem{corollary}[thm]{Corollary}
\newtheorem{proposition}[thm]{Proposition}
\newtheorem{thmdefi}[thm]{Theorem - Definition}
\newtheorem{step}[thm]{}
\newtheorem*{thm*}{Theorem}

\theoremstyle{definition}
\newtheorem{definition}[thm]{Definition}

\newtheorem{example}[thm]{Example}
\newtheorem{remark}[thm]{Remark}
\newtheorem{parg}[thm]{}
\newtheorem{notation}[thm]{Notation}
\newtheorem{construction}[thm]{Construction}
\newtheorem*{claim}{Claim}
\newtheorem{pargtwo}{}[thm]

\newcommand{\ph}{\varphi}
\newcommand{\w}{\widetilde}
\newcommand{\ma}{\mathcal}
\newcommand{\la}{\longrightarrow}
\newcommand{\ol}{\mathcal{O}}
\newcommand{\wi}{\widehat}
\newcommand{\pr}{\mathbb{P}}
\newcommand{\Q}{\mathbb{Q}}

\newcommand{\R}{\mathbb{R}}
\newcommand{\Z}{\mathbb{Z}}
\newcommand{\N}{\mathcal{N}_1}
\newcommand{\Nu}{\mathcal{N}^1}

\newcommand{\dom}{\operatorname{dom}}
\newcommand{\Sing}{\operatorname{Sing}}
\newcommand{\Pic}{\operatorname{Pic}}

\newcommand{\NE}{\operatorname{NE}}
\newcommand{\Exc}{\operatorname{Exc}}

\newcommand{\codim}{\operatorname{codim}}
\newcommand{\Eff}{\operatorname{Eff}}
\newcommand{\Nef}{\operatorname{Nef}}

\newcommand{\Mov}{\operatorname{Mov}}

\newcommand{\mov}{\operatorname{mov}}

\usepackage{etoolbox}
\patchcmd{\section}{\normalfont}{\normalfont\large}{}{}
\patchcmd{\subsection}{\bfseries}{\scshape\centering}{}{}
\patchcmd{\subsection}{-.5em}{.5em}{}{}

\title[Fano $4$-folds with rational fibrations]{Fano $4$-folds with rational fibrations}
\author{C.~Casagrande}
\address{Universit\`a di Torino,
Dipartimento di Matematica,
via Carlo Alberto 10,
10123 Torino - Italy}
\email{cinzia.casagrande@unito.it}
\date{September 17, 2019}
\calclayout 

\begin{document}
\maketitle
\renewcommand{\theequation}{\thethm}

\setcounter{tocdepth}{1}

{\footnotesize\tableofcontents}

\section{Introduction}
\noindent 
Smooth, complex Fano varieties play an important role in projective geometry, both from the classical and modern point of view, in the framework of the Minimal Model Program.
There are finitely many families of Fano varieties of any given dimension, which are classified up to dimension $3$ --
the classification of Fano $3$-folds was achieved more than 30 years ago, see \cite{fanoEMS} and references therein. In dimensions $4$ and higher there is no classification apart from some special classes, and we still lack a good understanding of the geometry of Fano $4$-folds. 

This paper is part of a program to study Fano $4$-folds $X$ with large Picard number $\rho_X$, by means of birational geometry, more precisely via the study of  contractions and flips of Fano $4$-folds. Our goal is to get a sharp bound on $\rho_X$, and possibly to classify Fano $4$-folds $X$ with ``large'' Picard number. Let us notice that, 
 among the known examples of Fano $4$-folds, products of del Pezzo surfaces have $\rho_X\leq 18$, and the others have $\rho_X\leq 9$ (see \cite{vb} for the case $\rho_X=9$).

In this paper we focus on Fano $4$-folds $X$ having a  rational contraction of fiber type. Here a \emph{contraction}  is a morphism $f\colon X\to Y$ with connected fibers onto a normal projective variety. More generally, a \emph{rational contraction} is a rational map $f\colon X\dasharrow Y$ 
that can be factored as $X\stackrel{\ph}{\dasharrow}X'\stackrel{f'}{\to}Y$, where $X'$ is a normal and $\Q$-factorial projective variety, $\ph$ is birational and an isomorphism in codimension $1$, and $f'$ is a 
contraction. As usual, $f$ is of fiber type if $\dim Y<\dim X$. Note that
  $X$ has
a non-constant rational contraction of fiber type if and only if there is a non-zero, non-big movable divisor.
Our main results are the following.
\begin{thm}\label{main}
Let $X$ be a smooth Fano $4$-fold with a  rational contraction of fiber type
 $f\colon X\dasharrow Y$, where $\dim Y>0$. 
If $Y\not\cong\pr^1$ and $Y\not\cong\pr^2$, then
$\rho_X\leq 18$, with equality only if $X$ is a product of surfaces.
\end{thm}
\begin{thm}\label{rationalfibration}
Let $X$ be a smooth Fano $4$-fold. Suppose that there exists a dominant rational map $f\colon X\dasharrow Y$, regular and proper on an open subset of $X$, with $\dim Y=3$. Then either $X$ is a product of surfaces, or $\rho_X\leq 12$.
\end{thm}
Let us say something on the techniques and strategy used in the paper. We consider the following classes of rational contractions of fiber type:
$$\{\text{``quasi-elementary''}\}\ \subset\  \{\text{``special''}\}\ 
\subset\ \{\text{general}\}.$$
Quasi-elementary rational contractions of fiber type have been introduced in \cite{fanos,eff} (see \S \ref{secquasiel} for more details); when $f$ is quasi-elementary Th.~\ref{main} is already known (ibidem), and  one can even allow $Y\cong\pr^1$ and $Y\cong\pr^2$.

 In this paper we introduce a more general notion, that of ``special'' rational contraction of fiber type, which plays a key role in the proof of Th.~\ref{main}. We define special (regular and rational) contractions
in \S \ref{secspecial};
then we show that every rational contraction of fiber type of a Mori dream space can be factored as a special rational contraction, followed by a birational map (Prop.~\ref{factorization1}). In particular, if a Fano variety has a  rational contraction of fiber type, then it also has a  special rational contraction of fiber type, so that we can reduce to prove Th.~\ref{main} when $f$ is special. 

Secondly, we show that
 up to flips, every special rational contraction of a Mori dream space can be factored as a sequence of elementary divisorial contractions, followed by a quasi-elementary contraction (Th.~\ref{factorization2}). This allows to relate the study of special rational contractions of Fano $4$-folds $X$ to our previous study  of elementary divisorial contractions and quasi-elementary contractions of $4$-folds obtained from $X$ with a sequence of flips, in \cite{eff,blowup}.

Another key ingredient used in the paper is the Lefschetz defect $\delta_X$, an invariant of $X$  which basically allows to bound $\rho_X$ in terms of the Picard number of prime divisors in $X$ (see \S \ref{prelLD} for an account).

After developing the necessary techniques and preliminary results in \S\S \ref{special} - \ref{sanfrancisco}, we prove Th.~\ref{main} first in the case where $\dim Y=2$ in \S \ref{sec_surf}, and then in the case where $\dim Y=3$ in \S \ref{sec_3folds}. Th.~\ref{rationalfibration} is then an easy consequence of the  case where $\dim Y=3$.

\medskip

\noindent{\bf Acknowledgments.} I am grateful to St\'ephane Druel for important suggestions.
\subsection{Notation and terminology}\label{terminology}
\noindent 
If $\mathcal{N}$ is a finite-dimensional real vector space and $a_1,\dotsc,a_r\in \mathcal{N}$, $\langle a_1,\dotsc,a_r\rangle$ denotes the convex cone in $\mathcal{N}$ generated by $a_1,\dotsc,a_r$.
Moreover, for every $a\neq 0$, $a^{\perp}$ is the hyperplane orthogonal to $a$ in the dual vector space $\ma{N}^*$.

We refer the reader to \cite{hukeel} for the notion of Mori dream space;  \emph{we always assume that a Mori dream space is projective, normal and $\Q$-factorial}. We recall that Fano varieties are Mori dream spaces by \cite[Cor.~1.3.2]{BCHM}.
We also refer
 to \cite{kollarmori} for the standard notions in birational geometry, in particular the definition of flip \cite[Def.~6.5]{kollarmori}

Let $X$ be a normal and $\Q$-factorial projective variety.

A small $\Q$-factorial modification (SQM) is a birational map $\ph\colon X \dasharrow X'$ which is an isomorphism in codimension one, where $X'$ is a normal and $\Q$-factorial projective variety. If $X$ is a Mori dream space, every SQM can be factored as a finite sequence of flips.

Let $f\colon X\to Y$ be an elementary contraction, namely a contraction with $\rho_X-\rho_Y=1$. We say that $f$ is of type $(a,b)$ if $\dim\Exc(f)=a$ and $\dim f(\Exc(f))=b$. We say that $f$ is of type 
$(\dim X-1,b)^{sm}$ if it is the blow-up of a smooth $b$-dimensional subvariety of $Y$, contained in $Y_{reg}$. If $X$ is a smooth $4$-fold, we say that $f$ is of type  $(3,0)^Q$ if $f$ is of type $(3,0)$, $\Exc(f)$ is isomorphic to an irreducible quadric $Q$, and $\mathcal{N}_{\Exc(f)/X}\cong \ol_Q(-1)$.

Let $D$ be a divisor.
A contraction $f\colon X\to Y$ is $D$-negative (respectively, $D$-positive) if there exists $m\in\Z_{>0}$ such that $-mD$ (respectively, $mD$) is Cartier and $f$-ample. A $D$-negative flip is the flip of a small, $D$-negative elementary contraction, and similarly for $D$-positive. \emph{We do not assume that contractions or flips are $K$-negative, unless specified.} 

When $X$ is a Mori dream space, given a contraction $f\colon X\to Y$ and a divisor $D$ in $X$, one can run a MMP for $D$ relative to $f$. This means that there exists a birational map $\psi\colon X\dasharrow X'$, given by a composition of $D$-negative flips and elementary divisorial contractions, such that $f':=f\circ\psi^{-1}\colon X'\to Y$ is regular, and
if $D'$ is the transform of $D$ in $X'$, then 
either $D'$ is $f'$-nef, or $f'$ factors through a $D'$-negative elementary contraction of fiber type of $X'$. 

A movable divisor is an effective divisor $D$ such that the stable
 base locus of the linear system $|D|$  has codimension $\geq 2$. 
A fixed prime divisor is a prime divisor $D$ which is the stable base locus of 
$|D|$. 
We will consider the usual cones of divisors and of curves: $$\Nef(X)\subseteq\Mov(X)\subseteq\Eff(X)\subset\Nu(X),\qquad
\mov(X)\subseteq \NE(X)\subset\N(X),$$ where 
all the notations are standard except $\mov(X)$, which is the convex cone generated by classes of curves moving in a family covering $X$.
When $X$ is a Mori dream space, all these cones are closed, rational and polyhedral. If $D$ is a divisor and $C$ is a curve in $X$, we denote by $[D]\in\Nu(X)$ and $[C]\in\N(X)$ their numerical equivalence classes.

For every closed subset $Z\subset X$, we denote by $\N(Z,X)$ the linear subspace of $\N(X)$ spanned by classes of curves contained in $Z$. We will use the following simple property.
\begin{remark}\label{caffe}
Let $D$ be a prime divisor. If $Z\cap D=\emptyset$, then $\N(Z,X)\subseteq D^{\perp}$, in particular $\N(Z,X)\subsetneq \N(X)$. This is because $D\cdot C=0$ for every curve $C\subset Z$.
\end{remark}
Let $X$ be a smooth $4$-fold. An \emph{exceptional plane} is a closed subset $L\subset X$ such that $L\cong\pr^2$ and $\ma{N}_{L/X}\cong\ol_{\pr^2}(-1)^{\oplus 2}$; an \emph{exceptional line} is a closed subset $\ell\subset X$ such that $\ell\cong\pr^1$ and $\ma{N}_{\ell/X}\cong\ol_{\pr^1}(-1)^{\oplus 3}$. 
\section{Special contractions of fiber type}\label{special}
\noindent When studying Fano varieties, or more generally Mori dream spaces, one often needs to consider contractions of fiber type $f\colon X\to Y$ which are not elementary. In full generality, such contractions are hard to deal with, in particular $Y$ may be very singular and/or non $\Q$-factorial. For this reason, it is useful to introduce some classes of contractions of fiber type with good properties, which should include the elementary case.
A first notion of this type is that of ``quasi-elementary'' contraction; we briefly recall this definition and some properties in \S \ref{secquasiel}.

 Here we introduce a more general notion, that of ``special'' contraction of fiber type. 
In \S \ref{secspecial}
we define special contractions, in the regular and rational case; the target  is automatically $\Q$-factorial. 

In \S \ref{secfactorizations} we show 
two factorization results for rational contractions of fiber type of Mori dream spaces. More precisely, we show
that every rational contraction of fiber type of a Mori dream space can be factored as a special rational contraction, followed by a birational map (Prop.~\ref{factorization1}). Moreover, up to flips, every special rational contraction of a Mori dream space can be factored as a sequence of elementary divisorial contractions, followed by a quasi-elementary contraction (Th.~\ref{factorization2}).

Finally, in \S \ref{secsingularities} we consider  special contractions of fiber type $f\colon X\to Y$ which are also $(K+\Delta)$-negative for a suitable boundary $\Delta$ on $X$, and we show that if $X$ has good singularities, then $Y$ has good singularities too.
\subsection{Quasi-elementary contractions}\label{secquasiel}
\noindent We refer the reader to \cite[\S 2.2]{eff} and \cite{fanos} for the notion of quasi-elementary  contraction of fiber type; here we just recall the definition.
\begin{definition}[quasi-elementary contraction]
Let $X$ be a normal and $\Q$-factorial projective variety and $f\colon X\to Y$ a contraction of fiber type. We say that $f$ is quasi-elementary if for every fiber $F$ of $f$ we have $\N(F,X)=\ker f_*$, where $f_*\colon\N(X)\to\N(Y)$ is the push-forward of one-cycles (see \S \ref{terminology} for $\N(F,X)$).
\end{definition}
Let us give an equivalent characterization, for Mori dream spaces.
\begin{proposition}\label{quasiel}
Let $X$ be a  Mori dream space  and $f\colon X\to Y$ a contraction of fiber type. 
The following are equivalent:
\begin{enumerate}[$(i)$]
\item $f$ is quasi-elementary;
\item for every prime divisor $D$ in $X$, either $f(D)=Y$, or $D=\lambda f^*B$ for some $\Q$-Cartier prime divisor $B$ in $Y$ and $\lambda\in\Q_{> 0}$;
\item $Y$ is $\Q$-factorial and for every prime divisor $B$ in $Y$, the pull-back $f^*B$ is irreducible (but possibly non-reduced).
\end{enumerate}
\end{proposition}
\begin{proof}
Let $F\subset X$ be a general fiber of $f$.

 $(i)\Rightarrow (iii)\ $ The target $Y$ is $\Q$-factorial by \cite[proof of Rem.~2.26]{eff}. Let $B$ be a prime divisor in $Y$, and let $D$ be an irreducible component of $f^*B$.  Then $D\cap F=\emptyset$, so that $\N(F,X)\subseteq D^{\perp}$ by Rem.~\ref{caffe}. Since $f$ is quasi-elementary, we have $\N(F,X)=\ker f_*$, hence 
$\ker f_*\subseteq D^{\perp}$, and  $D$ is the pull-back of a $\Q$-divisor in $Y$ (see \cite[Rem.~2.9]{eff}). Since $B=f(D)$, we must have  $D=\lambda f^*B$ with $\lambda\in\Q_{> 0}$, so $f^*B$ is irreducible.

\medskip

  $(ii)\Rightarrow (i)\ $
Let $\sigma$ be the minimal face of $\Eff(X)$ containing $f^*(\Nef(Y))$;  by \cite[Lemma 2.21 and Prop.~2.22]{eff} we have
 $\sigma=\Eff(X)\cap\N(F,X)^{\perp}$, and $f$ is quasi-elementary if and only if 
 $\dim\sigma=\rho_Y$.

Suppose that $f$ is not quasi-elementary.
Then $\dim\sigma>\rho_Y$, so that $\sigma\not\subseteq f^*\Nu(Y)$, and there exists a one-dimensional face $\tau$ of $\sigma$ such that $\tau \not\subseteq f^*\Nu(Y)$. Let $D\subset X$ be a prime divisor with $[D]\in\tau$. Then  $D$ is not the pull-back of a $\Q$-Cartier prime divisor in $Y$. 
On the other hand, we also have $[D]\in \N(F,X)^{\perp}$, so that
$D\cdot C=0$ for every curve $C\subset F$. Since $F\not\subset D$, we must have $F\cap D=\emptyset$, hence $f(D)\subsetneq Y$.

\medskip

  $(iii)\Rightarrow (ii)\ $ Let $D\subset X$ be a prime divisor which does not dominate $Y$. Let $B\subset Y$ be a prime divisor containing $f(D)$. Then $B$ is $\Q$-Cartier, and $D$ is an irreducible component of $f^*B$, hence $f^*B=\mu D$ with $\mu\in\Q_{> 0}$.
\end{proof}
\subsection{Special contractions}\label{secspecial}
\begin{definition}[special  contraction]
Let $X$ be a normal and $\Q$-factorial projective variety and $f\colon X\to Y$ a contraction of fiber type. We say that $f$ is special if for every prime divisor $D\subset X$ we have that either $f(D)=Y$, or $f(D)$ is a $\Q$-Cartier prime divisor in $Y$.
\end{definition}
\begin{remark}\label{Qfact}
Let $X$ be a normal and $\Q$-factorial projective variety and $f\colon X\to Y$ a contraction of fiber type. Then $f$ is special if and only if the following conditions hold:
\begin{enumerate}[(1)]
\item $\codim f(D)\leq 1$ for every prime divisor $D\subset X$;
\item $Y$ is $\Q$-factorial.
\end{enumerate}
\end{remark}
Condition (1) above is not enough to ensure that $Y$ is $\Q$-factorial, as the following simple example shows.
\begin{example}
Set $Z:=\pr_{\pr^2}(\ol\oplus\ol(1)\oplus\ol(1))$, $X:=Z\times\pr^1$, and let $\pi\colon X\to Z$ be the projection. Then $Z$ has a small elementary contraction $g\colon Z\to Y$, and $f:=g\circ\pi\colon X\to Y$ satisfies (1) but not (2), in particular it is not special. Note that $X$ is Fano and $f$ is $K$-negative.
\end{example}
\begin{remark}\label{unico}
Let $X$ be a normal and $\Q$-factorial projective variety and $f\colon X\to Y$ a contraction of fiber type.
\begin{enumerate}[$(a)$]
\item 
 If $X$ is a Mori dream space and $f$ is elementary, or quasi-elementary, then $f$ is special by Prop.~\ref{quasiel}.
\item\label{equidimensional}
If $f$ is special, then the locus where $f$ is not equidimensional has codimension at least $3$ in $Y$.  
\item\label{ind}
Let $f$ be  special, and $\ph\colon X\dasharrow X'$ a SQM such that $f':=f\circ\ph^{-1}$ is regular. Then $f'$ is special.
\end{enumerate}
\end{remark}
The following is a consequence of \cite[Lemma 2.6]{druelcodone}.
\begin{lemma}\label{kollar}
Let $X$ be a normal and $\Q$-factorial projective variety and $f\colon X\to Y$ a contraction of fiber type. If $f$ is equidimensional, then $Y$ is $\Q$-factorial and $f$ is special.
\end{lemma}
\begin{definition}[special rational contraction]
Let $X$ be a normal and $\Q$-factorial projective variety  and $f\colon X\dasharrow Y$ a rational contraction of fiber type.  We say that $f$ is  special if there
exists a SQM $\ph\colon X\dasharrow X'$ such that $f':=f\circ\ph^{-1}$ is regular and special. 
\end{definition}
\begin{remark}\label{Qfactrat} If $f\colon X\dasharrow Y$ is special, then:
\begin{enumerate}[--]
\item  $Y$ is $\Q$-factorial, by Rem.~\ref{Qfact};
\item 
 for every SQM $\ph\colon X\dasharrow X'$ such that $f':=f\circ\ph^{-1}$ is regular, we have that $f'$ is special, 
by Rem.~\ref{unico}$(\ref{ind})$. 
\end{enumerate}
\end{remark}
In the next subsection we will prove the following characterization of special rational contractions of Mori dream spaces.
\begin{proposition}\label{milano}
Let $X$ be a Mori dream space and $f\colon X\dasharrow Y$ a rational contraction of fiber type. Then $f$ is special if and only if $f$ cannot be factored as: 
$$\xymatrix{X\ar@{-->}[r]_{g}\ar@{-->}@/^1pc/[rr]^{f}&{Z}\ar@{-->}[r]_h&
{Y}
}$$
where $g$ is a rational contraction, $h$ is birational, and $\rho_Z>\rho_Y$.
\end{proposition}
\subsection{Factorizations}\label{secfactorizations}
\noindent We start this subsection with
a construction that will be used in the proofs of two factorization results,
Prop.~\ref{factorization1} and Th.~\ref{factorization2}.
\begin{construction}\label{constr}
Let $X$ be a Mori dream space, $f\colon X\to Y$ a contraction, and $D\subset X$  a prime divisor such that $f(D)\subsetneq Y$.
Let us run a MMP for $-D$, relative to $f$ (see \S \ref{terminology}). We get a commutative diagram:
\stepcounter{thm}
\begin{equation}\label{diagram}
\xymatrix{X\ar[d]_f\ar@{-->}[r]^{\psi}&{W}\ar[dl]_{f_W}\ar[d]^j
\\
Y& T\ar[l]^{k}
}\end{equation}
where $W$ is $\Q$-factorial, $\psi$ is a composition of $D$-positive flips and divisorial contractions (in particular $D$ cannot be exceptional for $\psi$, so it has a proper transform $D_W$ in $W$), and $f_W:=f\circ\psi^{-1}$ is regular. Since $f(D)\subsetneq Y$, the MMP cannot end with a fiber type contraction,  and $-D_W$ is $f_W$-nef. Let $j\colon W\to T$ be the  
   contraction  given by $\NE(f_W)\cap D_W^{\perp}$, so that $f_W$ factors as in \eqref{diagram}; there exists a $\Q$-Cartier prime divisor $D_T$ in $T$ such that $D_W=\lambda j^*D_T$ for some $\lambda\in\Q_{>0}$, and $-D_T$ is $k$-ample.
We have the following properties:
\begin{enumerate}[$(a)$]
\item $k$ is birational, $\Exc(k)\subseteq D_T$, $f(D)=k(D_T)$;
\item 
$f$, $f_W$, and $j$ coincide in the open subset $X\smallsetminus f^{-1}(f(D))$;
\item the divisorial irreducible components of $f^{-1}(f(D))$ are exactly $D$ and the prime exceptional divisors of $\psi$.
\begin{proof}
By construction $\psi$ is a composition of $D$-positive flips and divisorial contractions (relative to $f$), hence the images under $f$ of the exceptional divisors of $\psi$ are all contained in $f(D)$, so these divisors must be divisorial irreducible components of $f^{-1}(f(D))$. On the other hand
$k^{-1}(k(D_T))=D_T$, so $f_W^{-1}(f(D))=j^{-1}(D_T)=D_W$ is irreducible. 
\end{proof}
\item $f^{-1}(f(D))$ has 
$\rho_X-\rho_W+1$  divisorial irreducible components;
\item $k$ is an isomorphism if and only if $f(D)$ is a $\Q$-Cartier prime divisor in $Y$.
\begin{proof}
The ``only if'' direction is clear, because $D_T$ is $\Q$-Cartier and $f(D)=k(D_T)$. For the other, suppose that $f(D)$ is a $\Q$-Cartier prime divisor in $Y$. Since $k^{-1}(f(D))=k^{-1}(k(D_T))=D_T$, we must have
$k^*(f(D))=\mu D_T$, with $\mu\in\Q_{> 0}$. Then $-D_T$ is both $k$-trivial and
$k$-ample, so that $k$ must be an isomorphism.
\end{proof}
\item $\Exc(k)$ is a prime divisor if and only if $\codim f(D)>1$;
\item $k$ is not an isomorphism and $\codim\Exc(k)>1$  if and only if $f(D)$ is a non $\Q$-Cartier prime divisor.
\end{enumerate}
\end{construction}
\begin{proposition}\label{factorization1}
Let $X$ be a Mori dream space and $f\colon X\dasharrow Y$ a rational contraction of fiber type. Then $f$ can be factored as follows:
$$\xymatrix{X\ar@{-->}[r]_{g}\ar@{-->}@/^1pc/[rr]^{f}&{Z}\ar[r]_h&
{Y}
}$$
where $g$ is a special rational contraction, and $h$ is birational.
Moreover, such a factorization is unique up to composition with a SQM of $Z$.
\end{proposition}
\begin{proof}
To show existence of the factorization, we proceed by induction on $\rho_X-\rho_Y$. 

If $\rho_X-\rho_Y=1$, then $f$ is elementary and hence special, so the statement holds with  $g=f$ and $h=\text{Id}_Y$. 

For the general case, up to composing with a SQM of $X$, we can assume that $f$ is regular.
If $f$ is special, then as before the statement holds with $g=f$. Otherwise, there exists a prime divisor $D$ in $X$ such that $f(D)\subsetneq Y$ and $f(D)$ is not a $\Q$-Cartier divisor in $Y$.

We apply Construction \ref{constr} to $f$ and $D$.
We get a diagram as \eqref{diagram}, where $k$ is not an isomorphism by $(e)$, because 
$f(D)$ is not a $\Q$-Cartier divisor in $Y$; in particular $\rho_T>\rho_Y$.

The composition $\tilde{f}:=j\circ \psi\colon X\dasharrow T$ is a rational contraction of fiber type with $\rho_X-\rho_T<\rho_X-\rho_Y$; by the induction assumption, 
$\tilde{f}$ can be factored as follows:
$$\xymatrix{X\ar[d]_f\ar@{-->}[rd]_{\tilde{f}}\ar@{-->}[r]^{g}&{Z}\ar[d]^{\tilde{h}}\\
Y& T\ar[l]^{k}
}$$
where $g$ is a special rational contraction of fiber type, and $\tilde{h}$ is birational. Then $h:=k\circ \tilde{h}\colon Z\to Y$ is birational, so we have a factorization as in the statement.

To show uniqueness, suppose that $f$ has another factorization $X\stackrel{g'}{\dasharrow} Z'\stackrel{h'}{\to}Y$ with $g'$ special and $h'$ birational; notice that both $Z$ and $Z'$ are $\Q$-factorial by Rem.~\ref{Qfactrat}. We show that the birational map $\ph:=(h')^{-1}\circ h\colon Z\dasharrow Z'$ is a SQM.

Let $B\subset Z$ be a prime divisor. Up to composing $g$ and $g'$ with a SQM of $X$, we can assume that $g'\colon X\to Z'$ is regular. Let $D\subset X$ be a prime divisor dominating $B$ under $g$; then $g'(D)\subsetneq Z'$, and since $g'$ is special, $B':=g'(D)$ is a prime divisor in $Z'$. This means that $\ph$ does not contract $B$. Similarly, we see that $\ph^{-1}$ does not contract divisors, hence $\ph$ is a SQM.
\end{proof}
\begin{proof}[Proof of Prop.~\ref{milano}]
Suppose that $f$ is not special, and consider the factorization of $f$ given 
by Prop.~\ref{factorization1}. Then $h$ cannot be an isomorphism, thus $\rho_Z>\rho_Y$.

Conversely, suppose that $f$
has a factorization as in the statement. By applying Prop.~\ref{factorization1} to $g$, we get a factorization of $f$ as follows:
$$\xymatrix{X\ar@{-->}[r]_{g'}\ar@{-->}@/^1pc/[rrr]^{f}&{Z'}\ar[r]_{h'}&Z\ar@{-->}[r]_h&
{Y}
}$$
where $g'$ is special and $h'$ is birational. Thus $h\circ h'$ is birational with $\rho_{Z'}>\rho_Y$; by the uniqueness part of Prop.~\ref{factorization1}, $f$ is not special.
\end{proof}
\begin{notation}\label{notation}
Let $X$ be a Mori dream space and $f\colon X\to Y$ a special contraction; recall that $Y$ is $\Q$-factorial by Rem.~\ref{Qfact}. 
If $B$ is a prime divisor in $Y$, then every irreducible component of $f^*B$ must dominate $B$.
As the general fiber of $f$ is irreducible, there are at most finitely many prime divisors in $Y$ whose pullback to $X$ is reducible. We fix the notation $B_1,\dotsc,B_m$ for these divisors in $Y$, where $m\in\Z_{\geq 0}$, and we denote by $r_i\in\Z_{\geq 2}$ the number of irreducible components of $f^*B_i$, for $i=1,\dotsc,m$ (we ignore the multiplicities of these components, and ignore the possible prime divisors $B$ such that $f^*B$ is irreducible but nonreduced). Note that by Prop.~\ref{quasiel}, $f$ is quasi-elementary if and only if $m=0$.

Given  a special rational contraction $f\colon X\dasharrow Y$, we will use the same notation $B_1,\dotsc,B_m$ and $r_1,\dotsc,r_m$, with the obvious meaning.
\end{notation}
\begin{thm}\label{factorization2}
Let $X$ be a Mori dream space and $f\colon X\to Y$ a special contraction; notation as in \ref{notation}.
Let $E$ be the union of (arbitrarily chosen) $r_i-1$ components of  $f^*B_i$, for  $i=1,\dotsc,m$. Then
there is a factorization:
$$\xymatrix{X\ar[d]_f\ar@{-->}[r]^g&{X'}\ar[dl]^{f'}\\
Y&
}$$
where $X'$ is projective, normal, and $\Q$-factorial,  $g$ is birational with  $\Exc(g)=E$,\footnote{We denote by $\Exc(g)$ the closure in $X$ of the exceptional locus of $g$ in its domain.} 
the general fiber of $f$ is contained in the open subset where $g$ is an isomorphism,
and $f'$ is quasi-elementary.
\end{thm}
\begin{proof}
We proceed by induction on $\rho_X-\rho_Y$. 
 If $f$ is elementary, then it is quasi-elementary, so $E=\emptyset$ and the statement holds with $X'=X$ and $f'=f$.

Let us consider the general case. If $f$ is quasi-elementary, then again  the statement holds with $f'=f$. 

Suppose that $f$ is not quasi-elementary, so that $m\geq 1$ 
by Prop.~\ref{quasiel}, and consider the divisor 
$B_1\subset Y$. 
Let $D$ be the irreducible component of $f^*B_1$ not contained in $E$; 
we have $f(D)=B_1$ because
 $f$ is special. We apply Construction \ref{constr} to $f$ and $D$, and get a diagram:
$$\xymatrix{X\ar[d]_f\ar@{-->}[r]^{\psi}&{W}\ar[dl]^{f_W}
\\
Y&
}$$
where $W$ is $\Q$-factorial, $\psi$ is a sequence of $D$-positive flips and divisorial contractions, relative to $f$, and the general fiber of $f$ is contained in the open subset where $\psi$ is an isomorphism (by $(b)$). 
Moreover
$f_W^*B_1$ is irreducible (by $(e)$), and
the exceptional divisors of $\psi$ are all the components of $f^*B_1$ except $D$ (by $(c)$). 
In particular,  $r_1-1\geq 1$ elementary divisorial contractions occur in $\psi$,
 so $\rho_W<\rho_X$.
Clearly  $f_W$ is still special, and we conclude by applying the induction assumption to $f_W$.
\end{proof}
In particular, given a special contraction $f\colon X\to Y$ with general fiber $F$, one can bound $\rho_X$ in terms of $\rho_Y$, $\rho_F$, and the number of irreducible components of $f^*B_i$, $i=1,\dotsc,m$.
\begin{corollary}\label{boundrho}
Let $X$ be a Mori dream space, $f\colon X\to Y$ a special contraction,  and $F\subset X$ a general fiber of $f$. Notation as in \ref{notation}. Then $$\rho_X=\rho_Y+\dim\N(F,X)+
\sum_{i=1}^m(r_i-1)\leq\rho_Y+\rho_F+\sum_{i=1}^m(r_i-1).$$
\end{corollary}
For the proof of Cor.~\ref{boundrho} we need the following simple property.
\begin{lemma}\label{stephane}
Let $\ph\colon X\dasharrow X'$ be a birational map between normal and $\Q$-factorial projective varieties. Let 
 $T\subset X$ be a closed subset contained in the open subset where $\ph$ is an isomorphism, and set $T':=\ph(T)\subset X'$. Then $\dim\N(T,X)=\dim\N(T',X')$.
\end{lemma}
\begin{proof}
We note that $\N(T,X)$ is the quotient of the vector space of real $1$-cycles in $T$ by the subspace of 1-cycles $\gamma$ such that $\gamma\cdot D=0$ for every divisor $D$ in $X$, so it is determined by the image of the restriction map $\Nu(X)\to\Nu(T)$, and similarly for $\N(T',X')$. Since $X$ and $X'$ are $\Q$-factorial, and $T$ is contained in the open subset where $\ph$ is an isomorphism, it is easy to see that the images of the maps $\Nu(X)\to\Nu(T)$ and $\Nu(X')\to\Nu(T')$ are the same, under the natural isomorphism $\Nu(T)\cong\Nu(T')$.
\end{proof}
\begin{proof}[Proof of Cor.~\ref{boundrho}]
Let us consider the factorization of $f$ given by Th.~\ref{factorization2}. 
The difference $\rho_X-\rho_{X'}$ is the  number of prime exceptional divisors of $g$, namely $\sum_{i=1}^m(r_i-1)$. Moreover
 $F$ is contained in the open subset where $g$ is an isomorphism, $g(F)\subset X'$ is a general fiber of $f'$, and $\dim\N(F,X)=\dim\N(g(F),X')$ by Lemma \ref{stephane}. Finally, since $f'$ is quasi-elementary, we have $\rho_{X'}=\rho_Y+\dim\N(g(F),X')$. This yields the statement.
\end{proof}
\begin{corollary}\label{face}
Let $X$ be a Mori dream space and $f\colon X\to Y$ a special contraction; notation as in \ref{notation}.  Then every prime divisor in $f^*B_i$ is a fixed divisor, for $i=1,\dotsc,m$.

Moreover, let $E$ be the union of (arbitrarily chosen) $r_i-1$ components of  $f^*B_i$, for $i=1,\dotsc,m$.
Then the classes of the components of $E$ in $\Nu(X)$ generate a simplicial face $\sigma$ of $\Eff(X)$, and $\sigma\cap\Mov(X)=\{0\}$. 
\end{corollary}
\begin{proof}
Th.~\ref{factorization2} implies the existence of a contracting birational map $g\colon X\dasharrow X'$, with $X'$ $\Q$-factorial, whose prime exceptional divisors are precisely the components of $E$. This gives the statement (see for instance \cite[Lemma 2.7]{okawa_MCD}). 
\end{proof}
We will also need the following technical property.
\begin{lemma}\label{unosolo}
Let $X$ be a Mori dream space and $f\colon X\dasharrow Y$ a special rational contraction; notation as in \ref{notation}. Let $E_0$ be an irreducible component of $f^*B_i$ for some $i\in\{1,\dotsc,m\}$. Then there is a factorization of $f$:
$$\xymatrix{X\ar@{-->}[r]^{\ph}\ar@{-->}[d]_f&{\wi{X}}\ar[d]^{\sigma}\\
Y&Z\ar[l]
}$$
where $\ph$ is a SQM, $\sigma$ is an elementary divisorial contraction, 
 $\Exc(\sigma)$ is the transform of $E_0$, and $\dim\sigma(\Exc\sigma)\geq\dim Y-1$.
\end{lemma}
\begin{proof}
Let us choose a SQM $\psi\colon X\dasharrow X'$ such that $f':=f\circ\psi^{-1}\colon X'\to Y$ is regular.

We still denote by $E_0$ the transform of $E_0$ in $X'$;
by Cor.~\ref{face}, $E_0$ is a fixed divisor, and it is easy to see that it cannot be $f'$-nef. We run a MMP in $X'$ for $E_0$, relative to $f'$, and get a diagram:
$$\xymatrix{X\ar@{-->}[r]^{\psi}\ar@{-->}[dr]_f&{X'}\ar@{-->}[r]^{\xi}\ar[d]_{f'}&
{\wi{X}}\ar[d]^{\sigma} \\
& Y&Z\ar[l]_h
}$$
where $\xi$ is a sequence of $E_0$-negative flips, and $\sigma$ is an elementary divisorial contraction with exceptional divisor (the transform of) $E_0$.

Now $h\circ\sigma\colon\wi{X}\to Y$ is a special contraction, therefore $h(\sigma(\Exc(\sigma)))$ is a divisor in $Y$, and $\dim \sigma(\Exc(\sigma))\geq \dim Y-1$.
\end{proof}
\subsection{Singularities of the target}\label{secsingularities}
\noindent The goal of this subsection is to prove the following result.
\begin{proposition}\label{sing}
Let $X$ be a smooth projective variety, and $\Delta$ a $\Q$-divisor on $X$ such that $(X,\Delta)$ is klt. 
Let
 $f\colon X\to Y$ be a $(K+\Delta)$-negative special contraction of fiber type. 
 Then $Y$ has locally factorial, canonical singularities, and is nonsingular in codimension 2.
\end{proposition}
Prop.~\ref{sing} will follow from some technical lemmas.
\begin{lemma}\label{lf}
Let $X$ be a  projective variety  with locally factorial, canonical singularities, and  $\Delta$  a boundary such that $(X,\Delta)$ is klt. 
Let
 $f\colon X\to Y$ be a $(K+\Delta)$-negative special contraction of fiber type. 
 Then $Y$ has locally factorial, canonical singularities.
\end{lemma}
\begin{proof}
It follows from \cite[Cor.~4.5]{fujino} that 
 $Y$ has rational singularities, 
so it is enough to show that it is locally factorial \cite[Cor.~5.24]{kollarmori}. 

Let $B$ be a prime divisor in $Y$. Since $Y$ is $\Q$-factorial, there exists $m\in\Z_{>0}$ such that $mB$ is Cartier.

Set $U:=f^{-1}(Y_{reg})$; since $Y$ is normal and $f$ is special, we have $\codim \Sing(Y)\geq 2$ and $\codim(X\smallsetminus U)\geq 2$. Then $B\cap Y_{reg}$ is a Cartier divisor on $Y_{reg}$, and $f_{|U}^*(B\cap Y_{reg})$ is a Cartier divisor on $U$. Since $X$ is locally factorial, there exists a Cartier divisor $D$ in $X$ such that $D_{|U}=f_{|U}^*(B\cap Y_{reg})$. Then $(mD)_{|U}=f_{|U}^*((mB)_{|Y_{reg}})=f^*(mB)_{|U}$, and hence $mD=f^*(mB)$. 

We deduce that $D\cdot C=0$ for every curve $C\subset X$ contracted by $f$. Since $f$ is $(K+\Delta)$-negative, this implies that there exists a Cartier divisor $B'$ on $Y$ such that $D=f^*B'$ \cite[Th.~3.7(4)]{kollarmori}. 
Thus we have $B'_{|Y_{reg}}=B\cap Y_{reg}$, and hence $B=B'$ is Cartier.
\end{proof}
The following two lemmas are basically \cite[Prop.~1.4 and 1.4.1]{ABW}, where they are attributed to Fujita.
\begin{lemma}\label{quotient}
Let $X$ be a smooth projective variety, and $\Delta$ a $\Q$-divisor on $X$ such that $(X,\Delta)$ is klt. Let $f\colon X\to Y$ be an equidimensional, $(K+\Delta)$-negative contraction of fiber type. If $Y$ has at most finite quotient singularities, then $Y$ is smooth. 
\end{lemma}
\begin{proof}
Let $F\subset X$ be a general fiber of $f$. Then $F$ is smooth and $(F,\Delta_{|F})$ is klt \cite[Lemma 5.17]{kollarmori}; moreover $-(K_F+\Delta_{|F})\equiv -(K_X+\Delta)_{|F}$ is ample, so that $(F,\Delta_{|F})$ is log Fano. By Kawamata-Viehweg vanishing, $h^i(F,\ol_F)=0$ for every $i>0$, hence 
$\chi(F,\ol_F)=1$. Then the same proof as  \cite[Prop.~1.4]{ABW} applies.
\end{proof}
\begin{lemma}\label{target2}
Let $X$ be a smooth projective variety with $\dim X\geq 3$, and $\Delta$ a $\Q$-divisor on $X$  such that $(X,\Delta)$ is klt. Let $f\colon X\to S$ be an equidimensional, $(K+\Delta)$-negative contraction onto a surface. Then $S$ is smooth.
\end{lemma}
\begin{proof}
Notice first of all that $S$ is $\Q$-factorial by Lemma \ref{kollar}. Moreover, by \cite[Cor.~4.5]{fujino}, there exists  $\Q$-divisor $\Delta'$ on $S$ such that $(S,\Delta')$ is klt; in particular $S$ has log terminal  singularities, and hence finite quotient singularities \cite[Prop.~4.18]{kollarmori}. Then $S$ is smooth by Lemma \ref{quotient}.
\end{proof}
\begin{lemma}\label{cod3}
Let $X$ be a smooth projective variety, $\Delta$ a $\Q$-divisor on $X$ such that $(X,\Delta)$ is klt, and $f\colon X\to Y$ a $(K+\Delta)$-negative contraction of fiber type.

Suppose that the locus where $f$ is not equidimensional has codimension at least $3$ in $Y$, equivalently that there is no prime divisor $D\subset X$ such that $\codim f(D)=2$.

 Then $Y$ is smooth in codimension $2$.
\end{lemma}
\begin{proof}
Set $m=\dim Y$ and let $H_1,\dotsc,H_{m-2}$ be general very ample divisors in $Y$. Consider  $S:=H_1\cap\cdots\cap H_{m-2}$ and $Z:=f^{-1}(S)=f^*H_1\cap\cdots f^*H_{m-2}$. Then $S$ is a normal projective surface, $Z$ is smooth, and $f$ is equidimensional over $S$, so that $f_Z:=f_{|Z}\colon Z\to S$ is an equidimensional contraction.
Moreover $(Z,\Delta_{|Z})$ is klt \cite[Lemma 5.17]{kollarmori}.

Let $C\subset Z$ be a curve contracted by $f$; then $f^*H_i\cdot C=0$ for every $i$, so that by adjunction $(K_Z+\Delta_{|Z})\cdot C=(K_X+\Delta)\cdot C<0$, and $f_Z$ is $(K_Z+\Delta_{|Z})$-negative. Thus $S$ is smooth by Lemma \ref{target2}, so $S\subseteq Y_{reg}$ and hence $\codim\Sing Y\geq 3$.
\end{proof}
Prop.~\ref{sing} follows from Lemma \ref{lf}, Rem.~\ref{unico}$(\ref{equidimensional})$, and Lemma \ref{cod3}.
\section{Special contractions of Fano varieties of relative dimension 1}
\subsection{Preliminaries on the Lefschetz defect}\label{prelLD}
\noindent Let $X$ be a normal and $\Q$-factorial Fano variety. The {\em Lefschetz defect} 
$\delta_X$ is an invariant of $X$,  introduced in \cite{codim}, and defined as follows:
$$\delta_X=\max\left\{\codim\N(D,X)\,|\,D\text{ a prime divisor in }X\right\}$$
(see \S \ref{terminology} for $\N(D,X)$).
The main properties of $\delta_X$ are the following.
\begin{thm}[\cite{codim,gloria}]\label{trento}
Let $X$ be a $\Q$-factorial, Gorenstein Fano variety, with canonical singularities and at most finitely many non-terminal points. Then $\delta_X\leq 8$.

If moreover $X$ is smooth and $\delta_X\geq 4$, then $X\cong S\times Y$, where $S$ is a surface.
\end{thm}
\begin{thm}[\cite{codim}, Cor.~1.3 and \cite{cdue}, Th.~1.2]\label{buonconsiglio}
Let $X$ be a smooth Fano $4$-fold. Then one of the following holds:
\begin{enumerate}[$(i)$]
\item $X$ is a product of surfaces;
\item $\delta_X=3$ and $\rho_X\leq 6$;
\item $\delta_X=2$ and $\rho_X\leq 12$;
\item $\delta_X\leq 1$.
 \end{enumerate}
\end{thm}
\subsection{The case of relative dimension one}
\noindent In this subsection we show that if $X$ is a Fano variety and $f\colon X\to Y$ is a special contraction with $\dim Y=\dim X-1$, then $\rho_X-\rho_Y\leq 9$; this is a generalization of an analogous result in \cite{eleonora} in the case where $f$ is a conic bundle. The strategy of proof is the same: we use $f$ to produce $\rho_X-\rho_Y-1$ pairwise disjoint divisors in  $X$, and then we use them to show that if $\rho_X-\rho_Y\geq 3$, then $\delta_X\geq \rho_X-\rho_Y-1$; finally we apply Th.~\ref{trento}.
\begin{proposition}\label{reldim1}
Let $X$ be a $\Q$-factorial, Gorenstein Fano variety, with canonical singularities and at most finitely many non-terminal points. 
Let $f\colon X\to Y$ be a special contraction with $\dim Y=\dim X-1$. Then the following holds:
 \begin{enumerate}[$(a)$]
\item $\rho_X-\rho_Y\leq 9$;
\item  if $\rho_X-\rho_Y\geq 3$, then
 $\delta_X\geq\rho_X-\rho_Y-1$.
\end{enumerate}

If moreover $X$ is smooth and  $\rho_X-\rho_Y\geq 5$, then there exists a surface $S$ such that $X\cong S\times Z$, $Y\cong\pr^1\times Z$, and $f$ is induced by a conic bundle $S\to\pr^1$.
\end{proposition}
For the proof of Prop.~\ref{reldim1} we need some technical lemmas, that will be used also in \S~\ref{sec_3folds}.
\begin{lemma}\label{klt}
Let $X$ be a Mori dream space, and suppose that $K_X$ is Cartier in codimension $2$, namely 
that there exists a closed subset $T\subset X$ such that $\codim T\geq 3$ and $K_{X\smallsetminus T}$ is Cartier.

 Let $f\colon X\to Y$ be a $K$-negative special contraction with $\dim Y=\dim X-1$; notation as in \ref{notation}. Then $\rho_X=\rho_Y+1+m$ and $r_i=2$ for every $i=1,\dotsc,m$. 
 
 Let moreover $E_i,\wi{E}_i$ be the irreducible components of $f^*B_i$. Then the general fiber of $f$ over $B_i$ is $e_i+\hat{e}_i$, where $e_i$ and $\hat{e}_i$ are integral curves with $E_i\cdot e_i<0$, $\wi{E}_i \cdot\hat{e}_i<0$, and $-K_X\cdot e_i=-K_X\cdot \hat{e}_1=1$.
\end{lemma}
\begin{proof}
Fix $i\in\{1,\dotsc,m\}$. The closed subset $T$ cannot dominate $B_i$, hence
the general fiber of $f$ over $B_i$ is a curve $F_i$ contained in $X\smallsetminus T$ where $K_X$ is cartier.
Since $-K_X\cdot F_i=2$, and $f$ is $K$-negative,  $F_i$ has at most two irreducible components. This implies that $r_i=2$ and $F_i=e_i+\hat{e}_i$,
 with $e_i\subset E_i$, $\hat{e}_i\subset \wi{E}_i$, and conversely 
 $e_i\not\subset \wi{E}_i$, $\hat{e}_i\not\subset {E}_i$. The fiber $F_i$ is connected, hence we have $E_i\cap\hat{e}_i\neq\emptyset$, and therefore $E_i\cdot \hat{e}_i>0$. Since $E_i\cdot F_i=0$, we get $E_i\cdot e_i<0$; similarly for $\wi{E}_i$. Finally $\rho_X=\rho_Y+1+m$ by Cor.~\ref{boundrho}. 
\end{proof}
\begin{lemma}\label{gor}
In the setting of Lemma \ref{klt}, if moreover $\codim T\geq 4$, then $B_1,\dotsc,B_m$ are pairwise disjoint. 
\end{lemma}
\begin{proof}
By contradiction, suppose that $B_1\cap B_2\neq\emptyset$. Then $B_1\cap B_2$ has pure dimension $\dim X-3$, because $Y$ is $\Q$-factorial (see Rem.~\ref{Qfact}); let $W$ be an irreducible component. Since $f$ is special, the general fiber $F_W$ of $f$ over $W$ is a curve. Moreover,  $F_W$ is contained in the open subset where $K_X$ is cartier, so that $F_W=C+C'$ with $C$ and $C'$ integral curves of anticanonical degree $1$.

By Lemma \ref{klt}, for $i=1,2$ the general fiber $F_i$ of $f$ over $B_i$ is $e_i+\hat{e}_i$, with $-K_X\cdot e_i=1$, and $F_i$ degenerates to $F_W$. Thus, up to switching the components, we can assume that both $e_1$ and $e_2$ are numerically equivalent to $C$,  which implies that $e_1\equiv e_2$. This is impossible, because $E_1\neq E_2$, $E_i\cdot e_i<0$, and $e_i$ moves in a family of curves dominating $E_i$, for $i=1,2$.
\end{proof}
\begin{proof}[Proof of Prop.~\ref{reldim1}]
This the same as 
 the proof of \cite[Th.~1.1 and 1.3]{eleonora}, so we give only a sketch. We have $\rho_X=\rho_Y+1+m$ by Lemma \ref{klt}. As in \cite[Lemmas 3.9 and 3.10]{eleonora}, using Lemma \ref{gor}, one sees that if $m\geq 2$, 
then $\delta_X\geq m$. Hence the statement follows from Th.~\ref{trento}.
\end{proof}
\section{Preliminary results on Fano $4$-folds}\label{sanfrancisco}
\noindent From now on, we focus on smooth Fano $4$-folds. After giving in \S \ref{birat} some preliminary results on rational contractions of Fano $4$-folds, in \S \ref{secfixed} we recall the classification of fixed prime divisors in a Fano $4$-fold $X$ with $\rho_X\geq 7$, and report some properties that will be crucial 
in the sequel.
 Then in \S \ref{fabri} we apply the previous results to study special rational contractions of fiber type of $X$, when $\rho_X\geq 7$.
\subsection{Rational contractions of Fano $4$-folds}\label{birat}
\begin{lemma}[\cite{eff}, Rem.~3.6 and its proof]\label{basic1}
Let $X$ be a smooth Fano 
 $4$-fold and $\ph\colon X\dasharrow \w{X}$ a SQM.
\begin{enumerate}[$(a)$]
\item
 $\w{X}$ is smooth, the indeterminacy locus of $\ph$ is a disjoint union of exceptional planes (see \S \ref{terminology}), and the indeterminacy locus of
$\ph^{-1}$ is a disjoint union of exceptional lines;
\item an exceptional line in $\w{X}$ cannot meet any integral curve of anticanonical degree $1$, in particular it cannot meet an exceptional plane;
\item
let  $\psi\colon\w{X}\dasharrow\wi{X}$ be a SQM that factors as a sequence of $K$-negative flips. Then the indeterminacy locus of $\psi$ (respectively, $\psi^{-1}$) is a disjoint union of exceptional planes (respectively, lines).
\end{enumerate}
\end{lemma}
\begin{lemma}[\cite{eff}, Rem.~3.7]\label{basic2}
Let $X$ be a smooth Fano 
 $4$-fold and $f\colon X\dasharrow Y$ a rational contraction. Then 
one can factor $f$ as $X\stackrel{\ph}{\dasharrow}X'\stackrel{f'}{\to} Y$, where $\ph$ is a SQM, $X'$ is smooth, and $f'$ is a $K$-negative contraction.
\end{lemma}
These results allow to conclude that the target of a special rational contraction of a Fano $4$-fold has mild singularities.
\begin{lemma}\label{2}
Let $X$ be a smooth Fano $4$-fold and $f\colon X\dasharrow Y$ a special rational contraction. If $\dim Y=2$, then $Y$ is smooth. If $\dim Y=3$, then  $Y$ has isolated locally factorial, canonical singularities. 
\end{lemma}
\begin{proof}
By Lemma \ref{basic2} we can factor $f$ as $X\stackrel{\ph}{\dasharrow} X'\stackrel{f'}{\to} Y$ where $\ph$ is a SQM, $X'$ is smooth, and $f'$ is regular, $K$-negative, and special. Then the statement follows from Prop.~\ref{sing}.
\end{proof}
\subsection{Fixed prime divisors in Fano $4$-folds with $\rho\geq 7$}\label{secfixed}
\noindent Let $X$ be a Fano 
 $4$-fold with $\rho_X\geq 7$. 
Fixed prime divisors in $X$ have been classified in \cite{eff,blowup} in four types, and have many properties; this explicit information on the geometry of fixed divisors is a key ingredient in the proof of Th.~\ref{main}.
 In this subsection we recall this classification, and show some properties that will be used in the sequel.
\begin{thmdefi}[\cite{blowup}, Th.~5.1, Def.~5.3, Cor.~5.26, Def.~5.27]\label{long}
Let $X$ be a smooth Fano $4$-fold with $\rho_X\geq 7$, or $\rho_X=6$ and $\delta_X\leq 2$, and $D$ a fixed prime divisor in $X$. The following holds.
\begin{enumerate}[$(a)$]
\item Given a SQM $X\dasharrow X'$ and an elementary divisorial contraction $k\colon X'\to Y$ with $\Exc(k)$ the transform of $D$, then $k$ 
is of type $(3,0)^{sm}$, $(3,0)^Q$, $(3,1)^{sm}$, or $(3,2)$.
\item The type of $k$  depends only on $D$, so we 
define $D$ to be of type $(3,0)^{sm}$, $(3,0)^Q$, $(3,1)^{sm}$, or $(3,2)$, respectively.
\item If $D$ is of type $(3,2)$, then $D$ is the exceptional divisor of an elementary divisorial contraction of $X$, of type $(3,2)$.
\item We define $C_D\subset D\subset X$ to be the transform of a general irreducible curve $\Gamma\subset X'$ contracted by $k$, of minimal anticanonical degree; the curve $C_D$ depends only on $D$. 
\item  $C_D\cong\pr^1$, $D\cdot C_D=-1$, $C_D$ is contained in the open subset where the birational map $X\dasharrow X'$ is an isomorphism, and $C_D$ moves in a family of curves dominating $D$.
\item Let $\ph\colon X\dasharrow \w{X}$ be a SQM, and  $E$ a fixed prime divisor in $\w{X}$. We define the type of $E$ to be the type of its transform in $X$.
\end{enumerate}
 \end{thmdefi}
\noindent We will frequently use the notation $C_D\subset D$ introduced in the Theorem - Definition above.

The next property of fixed divisors of type $(3,2)$  will be crucial in the sequel.
\begin{lemma}\label{prop32}
Let $X$ be a smooth Fano $4$-fold with $\rho_X\geq 7$, or $\rho_X=6$ and $\delta_X\leq 2$, $X\dasharrow\w{X}$ a SQM, and 
$D\subset \w{X}$ a fixed divisor of type $(3,2)$.
If $\N(D,\w{X})\subsetneq\N(\w{X})$, then either $\rho_X\leq 12$, or $X$ is a product of surfaces.
\end{lemma}
\begin{proof}
If $\delta_X\geq 2$, we have the statement by Th.~\ref{buonconsiglio}, so let us assume that $\delta_X\leq 1$.
Let $D_X$ be the transform of $D$ in $X$, so that $D_X$ is the exceptional divisor of an elementary divisorial contraction of $X$, of type $(3,2)$. By \cite[Rem.~2.17(2)]{blowup}, $D_X$ cannot contain exceptional planes, hence $\dim\N(D_X,X)=\dim\N({D},\w{X})$ by \cite[Cor.~3.14]{eff}. Then $\rho_X\leq 12$ by \cite[Prop.~5.32]{blowup}.
\end{proof}
\begin{lemma}\label{2faces}
Let $X$ be a smooth Fano $4$-fold with $\rho_X\geq 7$, or $\rho_X=6$ and $\delta_X\leq 2$, and let $D_1,D_2\subset X$ be two distinct fixed prime divisors.
We have the following:

\smallskip 

\begin{enumerate}[$(a)$]
\item
$
\dim\langle [D_1],[D_2]\rangle\cap\Mov(X)=\dim\langle [C_{D_1}],[C_{D_2}]\rangle\cap\mov(X)=$\\

$\qquad\qquad\qquad\qquad=\begin{cases}
0\quad\text{ if }\ D_1\cdot C_{D_2}=0\text{ or }D_2\cdot C_{D_1}=0;\\
1\quad\text{ if }\ D_1\cdot C_{D_2}=D_2\cdot C_{D_1}=1;\\
2\quad\text{ if }\ (D_1\cdot C_{D_2})(D_2\cdot C_{D_1})\geq 2.\end{cases}$

\smallskip

\item If $D_1\cdot C_{D_2}=D_2\cdot C_{D_1}=1$, then $\langle [D_1],[D_2]\rangle\cap\Mov(X)=\langle[D_1+D_2]\rangle$ and
$\langle [C_{D_1}],[C_{D_2}]\rangle\cap\mov(X)=\langle[C_{D_1}+C_{D_2}]\rangle$. Moreover $(D_1+D_2)\cdot (C_{D_1}+C_{D_2})=0$ and  $D_1+D_2$ is not big.

\smallskip

\item If $D_1\cdot C_{D_2}=0$ or $D_2\cdot C_{D_1}=0$, then $\langle [D_1],[D_2]\rangle$ is a face of $\Eff(X)$, and $\langle [C_{D_1}],[C_{D_2}]\rangle$ is a face of $\Mov(X)^{\vee}$.
\end{enumerate}
\end{lemma}
For the proof, we  need the following elementary property in convex geometry.
\begin{lemma}\label{cone}
Let $\sigma$ be a convex polyhedral cone, of maximal dimension, in a finite dimensional real vector space $\ma{N}$. Let $\tau_1$ be a one-dimensional face of $\sigma$, and let $\alpha\in\ma{N}^{*}$ (the dual vector space) be such that
$\alpha\cdot\tau_1<0$ and $\alpha\cdot \eta\geq 0$ for every one-dimensional face $\eta\neq\tau_1$ of $\sigma$.

If $\tau_2$ is a one-dimensional face of $\sigma$ such that $\alpha\cdot\tau_2=0$, then $\tau_1+\tau_2$ is a face of $\sigma$.
\end{lemma}
\begin{proof}
Since $\tau_2$ is a face of $\sigma$, there exists $\beta\in\ma{N}^{*}$ such that $\beta\cdot x\geq 0$ for every $x\in\sigma$, and $\beta^{\perp}\cap\sigma=\tau_2$.
Let $y\in\tau_1$ be a non-zero element, and
set 
%$$a:=\alpha\cdot y\quad\text{and}\quad b:=\beta\cdot y.$$ 
$a:=\alpha\cdot y$ and $b:=\beta\cdot y$.
Then $a,b\in\R$, $a<0$, and $b>0$ (because $\tau_2\neq\tau_1$ by our assumptions). Let us consider $\gamma:=b\alpha+|a|\beta \in\ma{N}^{*}$. 

We have $\alpha\cdot\tau_2=\beta\cdot\tau_2=0$, hence $\gamma\cdot\tau_2=0$.
Moreover
$\gamma\cdot y=b\alpha\cdot y+|a|\beta\cdot y=0$,
namely $\gamma\cdot\tau_1=0$.
Finally if $\eta$ is a one-dimensional face of $\sigma$, different from $\tau_1$ and $\tau_2$, we have $\alpha\cdot\eta\geq 0$, $\beta\cdot\eta>0$, and hence $\gamma\cdot \eta>0$. 

Therefore $\gamma\cdot x\geq 0$ for every $x\in\sigma$, and $\gamma^{\perp}\cap\sigma=\tau_1+\tau_2$. This shows the statement. 
\end{proof}
\begin{proof}[Proof of Lemma \ref{2faces}]
We compute $\langle [D_1],[D_2]\rangle\cap\Mov(X)$.
Set $B:=\lambda_1D_1+\lambda_2D_2$ with $\lambda_i\in\R_{\geq 0}$ for $i=1,2$. By
\cite[Lemma 5.29(2)]{blowup}, $B$ is movable if and only if $B\cdot C_D\geq 0$ for every fixed prime divisor $D\subset X$, and this is equivalent to  $B\cdot C_{D_i}\geq 0$ for $i=1,2$, namely to:
\stepcounter{thm}
\begin{equation}\label{system}
\begin{cases}-\lambda_1+\lambda_2D_2\cdot C_{D_1}\geq 0\\
\lambda_1D_1\cdot C_{D_2}-\lambda_2\geq 0.\end{cases}\end{equation}
Let $\mathcal{S}\subseteq (\R_{\geq 0})^2$ be the set of non-negative solutions $(\lambda_1,\lambda_2)$ of \eqref{system}, so that $\mathcal{S}$ determines the intersection $\langle [D_1],[D_2]\rangle\cap\Mov(X)$. Notice that $(D_1\cdot C_{D_2})(D_2\cdot C_{D_1})$ is always non-negative, because $D_1\neq D_2$.
It is elementary to check that:
\begin{enumerate}[$\bullet$] 
\item $\mathcal{S}=\{(0,0)\}$ $\ \Longleftrightarrow\ $ $1-(D_1\cdot C_{D_2})(D_2\cdot C_{D_1})>0$ $\ \Longleftrightarrow\ $ $D_1\cdot C_{D_2}=0$ or $D_2\cdot C_{D_1}=0$;
\item
$\mathcal{S}$ is a   half-line
$\ \Longleftrightarrow\ $
 $1-(D_1\cdot C_{D_2})(D_2\cdot C_{D_1})=0$ $\ \Longleftrightarrow\ $  $D_1\cdot C_{D_2}=D_2\cdot C_{D_1}=1$, moreover in this case $\mathcal{S}=\{(\lambda,\lambda)\,|\,\lambda\geq0\}$;
\item $\mathcal{S}$ is a $2$-dimensional cone $\ \Longleftrightarrow\ $  $1-(D_1\cdot C_{D_2})(D_2\cdot C_{D_1})<0$$\ \Longleftrightarrow\ $  $(D_1\cdot C_{D_2})(D_2\cdot C_{D_1})\geq 2$.
\end{enumerate}

Similarly, we compute $\langle [C_{D_1}],[C_{D_2}]\rangle\cap\mov(X)$. We have
$$\mov(X)^{\vee}=\Eff(X)=\langle[D]\rangle_{D\text{\,fixed}}+\Mov(X).$$
Set $\gamma:=\lambda_1C_{D_1}+\lambda_2C_{D_2}$ with $\lambda_1,\lambda_2\in\R_{\geq 0}$. We have  $\gamma\cdot M\geq 0$ for every movable divisor $M$ in $X$ (see \cite[Lemma 5.29(2)]{blowup}). Hence
$\gamma\in\mov(X)$ 
if and only if $\gamma\cdot D\geq 0$ for every fixed prime divisor $D\subset X$,
and this is equivalent to $\gamma\cdot D_i\geq 0$ for $i=1,2$, namely to:
$$
\begin{cases}-\lambda_1+\lambda_2D_1\cdot C_{D_2}\geq 0\\
\lambda_1D_2\cdot C_{D_1}-\lambda_2\geq 0,\end{cases}$$
which is the same system as \eqref{system}, but with $\lambda_1$ and $\lambda_2$ interchanged. Thus the previous discussion yields $(a)$ and $(b)$.

We show $(c)$. Suppose for instance that $D_1\cdot C_{D_2}=0$. To see that $\langle [D_1],[D_2]\rangle$ is a face of $\Eff(X)$, we apply Lemma \ref{cone} with $\sigma=\Eff(X)$, $\tau_1=\langle[D_2]\rangle$, $\alpha=[C_{D_2}]$, and $\tau_2=\langle[D_1]\rangle$. It is enough to remark that $D\cdot C_{D_2}\geq 0$ for every prime divisor $D\neq D_2$.

Similarly, to see that $\langle [C_{D_1}],[C_{D_2}]\rangle$ is a face of $\Mov(X)^{\vee}$, we apply Lemma \ref{cone} with $\sigma=\Mov(X)^{\vee}$, $\tau_1=\langle[C_{D_1}]\rangle$, $\alpha=[D_1]$, and $\tau_2=\langle[C_{D_2}]\rangle$. Indeed $\langle[C_{D_1}]\rangle$ and 
$\langle[C_{D_2}]\rangle$ are one-dimensional faces of $\Mov(X)^{\vee}$ by \cite[Lemma 5.29(1)]{blowup}. Moreover $D_1\cdot\gamma\geq 0$ for every $\gamma\in\mov(X)$, and $D_1\cdot C_D\geq 0$ for every fixed prime divisor $D\neq D_1$. By \cite[Lemma 5.29(2)]{blowup} we have
$$\Mov(X)^{\vee}=\langle[C_D]\rangle_{D\text{\,fixed}}+\mov(X),$$
therefore $D_1\cdot\eta\geq 0$ for every one-dimensional face $\eta$ of $\Mov(X)^{\vee}$ different from $\langle[C_{D_1}]\rangle$. Thus the hypotheses of Lemma \ref{cone} are satisfied, and we get $(c)$.
\end{proof}
\begin{lemma}\label{description}
Let $X$ be a smooth Fano $4$-fold with $\rho_X\geq 7$, and let $D_1,D_2\subset X$ be two distinct fixed prime divisors such that $\langle[D_1],[D_2]\rangle\cap\Mov(X)=\{0\}$.
 Then, up to exchanging $D_1$ and $D_2$, one of the following holds:
\begin{enumerate}[$(a)$]
\item $D_1\cdot C_{D_2}=D_2\cdot C_{D_1}=0$ and $D_1\cap D_2=\emptyset$;
\item $D_1\cdot C_{D_2}=D_2\cdot C_{D_1}=0$ and $D_1\cap D_2$ is a disjoint union of exceptional planes;
\item $D_1\cdot C_{D_2}=D_2\cdot C_{D_1}=0$, $D_1$ is of type $(3,2)$, and $D_2$ is not of type $(3,0)^{sm}$;
\item  $D_1\cdot C_{D_2}>0$, $D_2\cdot C_{D_1}=0$, $D_1$ is of type  $(3,2)$,  and $D_2$ is of type $(3,1)^{sm}$ or $(3,0)^Q$.
\end{enumerate}
\end{lemma}
\begin{proof}
By \cite[Th.~5.1]{blowup} there  is a diagram
$$X\dasharrow \w{X}\stackrel{f}{\la}Y$$
where the first map is a SQM and $f$ is an elementary divisorial contraction with exceptional divisor the transform $\w{D}_2\subset\w{X}$ of $D_2$. Let $\w{D}_1\subset\w{X}$ be the transform of $D_1$.  By \cite[Lemma 2.21]{blowup}, $D_1$ is the transform of a fixed prime divisor $B_1\subset Y$. 

If $\w{D}_1\cap\w{D}_2=\emptyset$, then  $D_1\cap D_2$ is  contained in the indeterminacy locus of the map $X\dasharrow\w{X}$, which is a disjoint union of exceptional planes by Lemma \ref{basic1}$(a)$. Therefore either $D_1\cap D_2=\emptyset$ and we get $(a)$, or $D_1\cap D_2$ has pure dimension $2$ and we get
$(b)$.

We assume from now on that 
 $\w{D}_1\cap\w{D}_2\neq\emptyset$.

\medskip

Suppose that $D_2$ is of type $(3,1)^{sm}$. Then  $Y$ is a smooth Fano $4$-fold by \cite[Th.~5.1]{blowup}, $f$ is the blow-up of a smooth curve $C\subset Y$, and
  $B_1\cap C\neq\emptyset$. Then \cite[Lemma 5.11]{blowup} yields that $B_1$ is the exceptional divisor of an elementary divisorial contraction of type $(3,2)$, and either $B_1\cdot C>0$, or $B_1\cdot C<0$. Thus $B_1$ is generically a $\pr^1$-bundle over a surface, and the general fiber $F$ of this $\pr^1$-bundle satisfies $B_1\cdot F=K_Y\cdot F=-1$.
  Using Lemma \ref{basic1}$(a)$ and \cite[Lemma 2.18]{blowup}, one sees that
$D_1$ must be of type $(3,2)$. Moreover $C\cap F=\emptyset$ implies that $\w{D}_2$ is disjoint from the transform $\w{F}$ of $F$ in $\w{X}$, and $\w{D}_1$ is still generically a $\pr^1$-bundle with fiber $\w{F}$. The indeterminacy locus of the map $\w{X}\dasharrow X$ has dimension at most one (see Lemma \ref{basic1}$(a)$), hence $\w{F}$ is contained in the open subset where this map is an isomorphism, and in $X$ we get
$D_2\cdot C_{D_1}=\w{D}_2\cdot \w{F}=0$. Finally it 
 is easy to check that $D_1\cdot C_{D_2}=0$ if  $B_1\cdot C>0$ (and we have $(c)$), while 
$D_1\cdot C_{D_2}>0$ if  $B_1\cdot C<0$ (and we have $(d)$). So we get the statement.

\medskip

We can assume now that neither $D_1$ nor $D_2$ are of type $(3,1)^{sm}$.
Suppose  that $D_2$ is of type $(3,0)^{sm}$ or $(3,0)^Q$. Then $\w{D}_2$ is isomorphic to $\pr^3$ or to an irreducible quadric; let $\Gamma\subset \w{D}_2$ be a curve corresponding to a line. We have 
$\w{D}_1\cdot\Gamma>0$, and since $\Gamma$ is contained in the open subset where the map $\w{X}\dasharrow X$ is an isomorphism (see Th.-Def.~\ref{long}$(e)$), we also have
 $D_1\cdot C_{D_2}>0$. This yields $D_2\cdot C_{D_1}=0$ by Lemma \ref{2faces}. Therefore $D_1$ cannot be of type $(3,0)^{sm}$ nor $(3,0)^Q$, and the only possibility is that $D_1$ is of type $(3,2)$. 
Moreover, since $f(\w{D}_2)$ is contained in $B_1$, \cite[Lemma 5.41]{blowup} yields that $D_2$ cannot be of type $(3,0)^{sm}$, so we get again $(d)$.

\medskip

We are left with the case where both $D_1$ and $D_2$ are of type $(3,2)$, and we can assume that $D_1\cdot C_{D_2}=0$ by Lemma \ref{2faces}. If $\delta_X\geq 3$, then Th.~\ref{buonconsiglio} implies that $X$ is a product of surfaces; in this case it is easy to check directly that  $D_2\cdot C_{D_1}=0$. If $\delta_X\leq 2$, then we get  $D_2\cdot C_{D_1}=0$ by \cite[Lemma 2.2(b)]{cdue}. So we have $(c)$.
\end{proof}
\subsection{Special rational contractions of Fano $4$-folds with $\rho_X\geq 7$}\label{fabri}
\noindent Given a Fano $4$-fold $X$ with $\rho_X\geq 7$, and a  special rational contraction of fiber type
 $f\colon X\dasharrow Y$, in this subsection we show 
 that, for every  prime divisor $B$ of $Y$, $f^*B$ has at most two irreducible components. Moreover  we 
give conditions on the type of the fixed prime divisors in $f^*B$, when $f^*B$ is reducible.
\begin{lemma}\label{montenero}
Let $X$ be a smooth Fano $4$-fold with $\rho_X\geq 7$, or $\rho_X=6$ and $\delta_X\leq 2$, and $f\colon X\dasharrow Y$ a special rational contraction; notation as in \ref{notation}. Let $i\in\{1,\dotsc,m\}$.

If $\dim Y=3$, then every fixed divisor in $f^*B_i$ is of type $(3,2)$.

If $\dim Y=2$, then every fixed divisor in $f^*B_i$ is of type $(3,2)$ or $(3,1)^{sm}$.
\end{lemma}
\begin{proof}
Let $E_0$ be an irreducible component of $f^*B_i$. By Lemma \ref{unosolo} there are a SQM $X\dasharrow\w{X}$ and an elementary divisorial contraction $\sigma\colon\w{X}\to Z$ such that $\Exc(\sigma)$ is the transform of $E_0$, and $\dim\sigma(\Exc(\sigma))\geq\dim Y-1$. Th.-Def.~\ref{long} yields the statement.
\end{proof}
\begin{lemma}\label{r_i=2}
Let $X$ be a smooth Fano $4$-fold with $\rho_X\geq 7$, and $f\colon X\dasharrow Y$ a special rational contraction; notation as in \ref{notation}. Then $r_i=2$ for every $i=1,\dotsc,m$.
\end{lemma}
\begin{proof}
We consider for simplicity $i=1$.
\begin{claim}
For every irreducible component $D$ of $f^*B_1$, there exists another component $E$ of $f^*B_1$ such that $E\cdot C_D>0$. 
\end{claim}
Let us first show that the Claim implies the statement. Assume by contradiction that $r_1>2$, and let us consider a component $D_1$ of $f^*B_1$. By the Claim, there exists a second component $D_2$ with $D_2\cdot C_{D_1}>0$, and since $r_1\geq 3$, we have
$\langle [D_1],[D_2]\rangle\cap\Mov(X)=\{0\}$
by Cor.~\ref{face}. Applying Lemma \ref{description},
 we conclude that $D_1$ is not of type $(3,2)$, and $D_2$ is of type $(3,2)$. 

Now we restart with $D_2$, and we deduce that $D_2$ is not of type $(3,2)$, a contradiction. Hence $r_1=2$.

\medskip

We prove the Claim. By Lemma \ref{unosolo}, there exists a diagram:
$$\xymatrix{{X}\ar@{-->}[d]_{f}\ar@{-->}[r]^{\ph}&{\w{X}}\ar[d]^{\sigma}
\\
Y& {Z}\ar[l]^{g}
}$$
where $\ph$ is a SQM and $\sigma$ is an elementary divisorial contraction with $\Exc(\sigma)=\w{D}$, the transform of $D$ in $\w{X}$. 

Since $g\circ\sigma$ is special, we have $g(\sigma(\w{D}))=B_1$ and hence
 $\sigma(\w{D})\subset g^{-1}(B_1)$; let $E_Z\subset Z$ be an irreducible component of $g^{-1}(B_1)$ containing $\sigma(\w{D})$. Let $\w{E}\subset \w{X}$ and $E\subset X$ be the transforms of $E_Z$, so that $E$ is an irreducible component of $f^*B_1$. Note that $\w{E}\cdot\NE(\sigma)>0$ by construction.

Now let $\Gamma\subset \w{D}$ be a general minimal irreducible curve contracted by $\sigma$; by Th.-Def.~\ref{long}$(d)$ and $(e)$, the transform of $\Gamma$ in $X$ is the curve $C_D$, and $\Gamma$ is contained in the open subset where $\ph^{-1}\colon\w{X}\dasharrow X$ is an isomorphism.  Therefore
 $E\cdot C_D=\w{E}\cdot\Gamma>0$.
\end{proof}
\section{Fano $4$-folds to surfaces}\label{sec_surf}
\noindent In this section we study rational contractions from a Fano $4$-fold to a surface, and show the following.
\begin{thm}\label{surf}
Let $X$ be a smooth Fano $4$-fold having a rational contraction $f\colon X\dasharrow S$ with $\dim S=2$.
Then one of the following holds:
\begin{enumerate}[$(i)$]
\item $X$ is a product of surfaces;
\item
  $\rho_X\leq 12$;
\item
  $13\leq \rho_X\leq 17$,  $S$ is a smooth del Pezzo surface, 
the general fiber $F$ of $f$ is a smooth del Pezzo surface with $4\leq\dim\N(F,X)\leq \rho_F\leq 8$,
and $\rho_X\leq 9+\dim\N(F,X)$.
\item
$S\cong\pr^2$ and $f$ is special.
\end{enumerate}
\end{thm}
\begin{lemma}\label{chitarre}
Let $X$ be a smooth Fano $4$-fold with $\rho_X\geq 7$, and $f\colon X\dasharrow S$ a special rational contraction with $\dim S=2$; notation as in \ref{notation}. Then for every $i=1,\dotsc, m$ the divisor $f^*B_i$ has two irreducible  components, one a fixed divisor of type $(3,2)$, and  the other one of type $(3,2)$ or  $(3,1)^{sm}$.
\end{lemma}
\begin{proof}
We consider for simplicity $i=1$. By Lemma \ref{r_i=2} $f^*B_1$  has two irreducible components, and by Lemma \ref{montenero} they are of type  $(3,2)$ or  $(3,1)^{sm}$. We have to show that they cannot be both of type $(3,1)^{sm}$.

Let us choose a SQM $\ph\colon X\dasharrow\w{X}$ such that $\tilde{f}:=f\circ\ph^{-1}\colon \w{X}\to S$ is regular, $K$-negative, and special (see Lemma \ref{basic2}). Let $E,\wi{E}\subset\w{X}$ be the irreducible components of $\tilde{f}^*(B_1)$, and
 $F\subset\w{X}$ a general fiber of $\tilde{f}$ over the curve $B_1$.

Suppose that $E$ is of type $(3,1)^{sm}$.
By Lemma \ref{unosolo} and Th.-Def.~\ref{long}, 
we have a diagram:
$$\xymatrix{X\ar@{-->}[rrd]_f\ar@{-->}[r]^{\ph}&{\w{X}}\ar[dr]^{\tilde{f}}\ar@{-->}[r]^{\psi}&{\wi{X}}
\ar[d]^{\hat{f}}\ar[r]^k&{\w{X}_1}\ar[dl]^{f_1}\\
&&S&
}$$
where $\psi$ is SQM and $k$ is the blow-up of a smooth irreducible curve $C\subset \w{X}_1$, with exceptional divisor the transform  of $E\subset\w{X}$, and $f_1(C)=B_1$.

Recall from the proof of Lemma \ref{unosolo} that $\psi$ arises from a MMP for $E$, relative to $\tilde{f}$. Since $\tilde{f}$ is $K$-negative, one can use a MMP with scaling of $-K_{\w{X}}$ (see \cite[\S 3.10]{BCHM}, and for this specific case \cite[Prop.~2.4]{codim} which can be adapted to the relative setting), so that $\psi$ factors as a sequence of $K$-negative flips, relative to $\tilde{f}$. Then by Lemma \ref{basic1}$(b)$ and $(c)$, the indeterminacy locus of $\psi$ is a disjoint union of exceptional planes, and is disjoint from the indeterminacy locus of $\ph^{-1}$.

In particular, 
the indeterminacy locus of $\psi$  is contracted to 
 points by $\tilde{f}$. Since $F$ is a general fiber of $\tilde{f}$ over $B_1$, it must be contained in the open subset where $\psi$ is an isomorphism, and $\wi{F}:=\psi(F)\subset\wi{X}$ is a general fiber of $\hat{f}$ over $B_1$. 
 We also note that $F$ is contained in the open subset where $\ph^{-1}$ is an isomorphism: otherwise there should be an exceptional line contained in $E$, and this would give an exceptional line contained in
 $\Exc(k)$, contradicting \cite[Rem.~5.6]{blowup}. 

Every irreducible component of $\Exc(k)\cap\wi{F}$ is a fiber of $k$ over $C$.  We deduce that 
the  transform in $X$ of any curve in $E\cap F$ has class in $\R_{\geq 0}[C_{E}]$.

We have $\dim F\cap E\cap \wi{E}\geq 1$, let $\Gamma$ be an irreducible curve in $F\cap E\cap \wi{E}$. If $\wi{E}$ were of type $(3,1)^{sm}$ too, the transform of $\Gamma$ in $X$ should have class in both $\R_{\geq 0}[C_{E}]$ and
$\R_{\geq 0}[C_{\wi{E}}]$.
 This would imply that the classes of $C_{E}$ and $C_{\wi{E}}$ are proportional, and this is impossible by Th.-Def.~\ref{long}$(e)$.
 Therefore $E$ and $\wi{E}$ cannot be both of type $(3,1)^{sm}$.
\end{proof}
\begin{proof}[Proof of Th.~\ref{surf}]
We can assume that $\rho_X\geq 13$, otherwise we have $(ii)$.

By Prop.~\ref{factorization1} $f$ factors as a special rational contraction $g\colon X\dasharrow T$ followed by a birational map $T\to S$. 
There exists a SQM $\ph\colon X\dasharrow \w{X}$ such that $\w{X}$ is smooth and the composition $\tilde{g}:=g\circ\ph^{-1}\colon \w{X}\to T$ is regular, $K$-negative and special (see Lemma \ref{basic2}); in particular $T$ is a smooth surface by Lemma \ref{2}.
$$\xymatrix{X\ar@{-->}[r]^{\ph}\ar@{-->}[d]_f\ar@{-->}[dr]^{g}  &{\w{X}}\ar[d]^{\tilde{g}}\\S&T\ar[l]
}$$
 Finally $g$ has $r_i=2$ for every $i=1,\dotsc,m$ (notation as in \ref{notation}) by Lemma \ref{r_i=2}.

\medskip

Suppose that $m=0$, equivalently that $\tilde{g}$ is quasi-elementary. 
If $g$ is regular, then \cite[Th.~1.1(i)]{fanos} together with $\rho_X\geq 13$ yield that $X$ is a product of surfaces, so we have $(i)$.

Assume instead that $g$ is not regular, and let $F\subset X$ be a general fiber of $f$, which is also a general fiber of $g$. Since the indeterminacy locus of $\ph^{-1}$ has dimension $1$ (see Lemma \ref{basic1}$(a)$), it does not meet a general fiber of $\tilde{g}$. This means
that $F$ is contained in the open subset where $\ph$ is an isomorphism, and $\ph(F)$ is a general fiber of $\tilde{g}$. By Lemma \ref{stephane} and \cite[Cor.~3.9 and its proof]{eff} we have that $F$ is a smooth del Pezzo surface with $\rho_F\leq 8$ and 
$$\rho_X=\dim\N(F,X)+\rho_T\leq\rho_F+\rho_T\leq 8+\rho_T.$$
In particular $\rho_T\geq 13-8=5$.
Then \cite[Prop.~4.1 and its proof]{eff} imply that $g$ is not elementary and that $T$ is a del Pezzo surface. Therefore $\rho_X\leq 17$, $\dim\N(F,X)=\rho_X-\rho_T\geq 13-9=4$,
and
$S$ is a smooth del Pezzo surface too.
So we have $(iii)$.

\medskip

Suppose now that $m\geq 1$.
By Lemma \ref{chitarre},
 $(\tilde{g})^*B_1$ has an irreducible component $E$ 
which is a fixed divisor 
 of type $(3,2)$.
We have $(\tilde{g})_*\N(E,\w{X})=\R[B_1]$, so that $\codim\N(E,\w{X})\geq\rho_T-1$. If $\rho_T>1$, then we get $(i)$  by Lemma \ref{prop32}. 

Let us assume that $\rho_T=1$. Then $T\cong\pr^2$, because $T$ is a smooth rational surface. Moreover the birational map $T\to S$ must be an isomorphism, hence $S\cong\pr^2$ and $f$ is special, and we get $(iv)$.
\end{proof}
\section{Fano $4$-folds to $3$-folds}\label{sec_3folds}
\noindent In this section we study rational contractions from a Fano $4$-fold to a $3$-dimensional target, and show the following.
\begin{thm}\label{3folds}
Let $X$ be a smooth Fano $4$-fold. If there exists  a rational contraction $X\dasharrow Y$ with $\dim Y=3$, then  either $X$ is a product of surfaces, or $\rho_X\leq 12$.
\end{thm}
\begin{proof}
If $\delta_X\geq 3$ the statement follows from Th.~\ref{buonconsiglio}, so
we can assume that $\delta_X\leq 2$; we also assume that $\rho_X\geq 7$.
By Prop.~\ref{factorization1}, we can suppose that the map $X\dasharrow Y$ is special. Moreover by Lemma \ref{basic2} we can factor it as 
$$X\stackrel{\ph}{\dasharrow}\w{X}\stackrel{f}{\la}Y,$$
where $\ph$ is a SQM, $\w{X}$ is smooth, and $f$ is regular, $K$-negative and special. 

By Lemmas \ref{klt} and \ref{gor} we have $\rho_X=\rho_Y+m+1$, $r_1=\cdots=r_m=2$, and
 the divisors $B_1,\dotsc,B_m$ are pairwise disjoint in $Y$  (notation as in \ref{notation}). 
For $i=1,\dotsc,m$ the irreducible components of 
 $f^*B_i$ are fixed divisors of type $(3,2)$ by Lemma \ref{montenero}. 

If $\rho_X-\rho_Y\geq 3$, then $m\geq 2$. Let $E_1,E_2$ be the irreducible components of $f^*B_1$, and $W$ an irreducible component of $f^*B_2$.
Since $B_1\cap B_2=\emptyset$, we have $E_1\cap W=\emptyset$, so that 
$\N(E_1,\w{X})\subsetneq\N(\w{X})$ by Rem.~\ref{caffe}, and
this implies the statement by Lemma \ref{prop32}.

If instead $\rho_X-\rho_Y=1$, then $f$ is elementary, and $\rho_X\leq 11$ by \cite[Th.~1.1]{eff}. 

\bigskip

We are left with the case where
 $\rho_X-\rho_Y=2$ and $m=1$, which we assume from now on. We will adapt the proof of \cite[Th.~1.1]{eff} of the elementary case to the case $\rho_X-\rho_Y=2$, and divide the proof in several steps. Since $m=1$,  we
set for simplicity $B:=B_1$.
\begin{parg}\label{prel}
If $\N(E_1,\w{X})\subsetneq\N(\w{X})$ we conclude as before, so we can assume that $\N(E_1,\w{X})=\N(\w{X})$; this implies that $\N(B,Y)=\N(Y)$. 

By Lemma \ref{klt}, $E_1\cup E_2$ is covered by curves of anticanonical degree $1$. Since an exceptional line cannot meet such curves (see Lemma \ref{basic1}$(b)$), we deduce that
$\ell\cap (E_1\cup E_2)=\emptyset$ for every exceptional line $\ell\subset\w{X}$. 

Notice that even if $f$ is not elementary, by speciality it does not have fibers of dimension $3$, and has at most isolated fibers of dimension $2$. Moreover
$Y$ is locally factorial and has (at most) isolated canonical singularities, by Lemma \ref{2}. 
More precisely, $\Sing(Y)$ is contained in the images of the $2$-dimensional fibers of $f$ (this is due to Ando, see \cite[Th.~4.1 and references therein]{AWaview}).

Since $\w{X}$ is smooth and $Y$ is locally factorial, it is easy to see that $f^*B=E_1+E_2$.

Finally, since $X$ is Fano, by \cite[Lemma 2.8]{prokshok} there exists a $\Q$-divisor $\Delta_Y$ on $Y$ such that $(Y,\Delta_Y)$ is a klt
 log Fano, so that $-K_Y$ is big.
\end{parg}
\begin{step}\label{small}
Let $g\colon Y\to Y_0$ be a small elementary contraction. Then $\Exc(g)$ is the disjoint union of smooth rational curves lying in the smooth locus of $Y$, with normal bundle $\ol_{\pr^1}(-1)^{\oplus 2}$; in particular $K_Y\cdot\NE(g)=0$. 
\end{step}
\begin{proof}
Exactly the same proof  as the one of \cite[Lemma 4.5]{eff} applies, with the only difference that, in the notation of 
\cite[Lemma 4.5]{eff}, $\dim\N(\w{U}/U)$ could be bigger than $2$. We take $\tau$ to be any extremal ray of $\NE(\w{U}/U)$ 
 not contained $\NE(g_{|\w{U}})$. 
\end{proof}
\begin{step}\label{divisorial}
Let $g\colon Y\to Y_0$ be an elementary divisorial contraction. Then $g$ is the blow-up of a smooth point of $Y_0$; in particular $-K_Y\cdot\NE(g)>0$.
\end{step}
\begin{proof}
Set $G:=\Exc(g)\subset Y$.
Since $g$ is elementary and $\dim g(G)\leq 1$, we have $\dim\N(G,Y)\leq 2$; on the other hand  $\dim\N(B,Y)=\rho_Y=\rho_X-2\geq 5$ (see \ref{prel}), so $G\neq B$, and  $D:=f^*G$ is a prime divisor in $\w{X}$, different from $E_1$ and $E_2$, with $\dim\N(D,\w{X})\leq \dim\ker f_*+\dim\N(G,Y)\leq 2+2=4$.

Since $G$ is fixed, also $D$ is a fixed divisor in $\w{X}$; let  $D_X\subset X$ be the transform of $D$. 
\begin{pargtwo}\label{no32}
We show that $D$ is not of 
of type $(3,2)$. 
Otherwise, as in the proof of Lemma \ref{prop32} we see that
$\dim\N(D_X,X)=\dim\N({D},\w{X})\leq 4$. On the other hand we have $\delta_X\leq 2$ and $\rho_X\geq 7$, a contradiction.
\end{pargtwo}
\begin{pargtwo} We show that $g$ is of type $(2,0)$. By contradiction, suppose that $g$ is of type $(2,1)$. As in  \cite[proof of Lemma 4.6]{eff}, we show that there is an open subset $\w{U}\subseteq\w{X}$ such that $D\cap\w{U}$ is covered by curves of anticanonical degree $1$. 
By \cite[Lemma 2.8(3)]{blowup}, $D_X$ still has a non-empty open subset covered by curves of anticanonical degree $1$; this implies that $D_X$ and $D$ are of type $(3,2)$ by \cite[Lemma 2.18]{blowup}, a contradiction to \ref{no32}.
\end{pargtwo}
\begin{pargtwo}\label{frankfurt}
Thus $g$ is of type $(2,0)$; set $p:=g(G)\in Y_0$. 

Since $\N(B,Y)=\N(Y)$ by \ref{prel},
we must have $G\cap B\neq \emptyset$ by Rem.~\ref{caffe}. 
Therefore $p\in g(B)$, hence $g^*(g(B))=B+aG$ with $a>0$, and $(g\circ f)^*(g(B))
=E_1+E_2+aD$ (see again \ref{prel}).

As in \cite[proof of Lemma 4.6]{eff}, we get a diagram:
$$\xymatrix{{\w{X}}\ar[d]^f\ar@{-->}[r]^{\psi}&{\wi{X}}\ar[r]^k&{\w{X}_1}\ar[dl]^{f_1}
\\Y\ar[r]^g&{Y_0}&
}$$
where $\psi$ is a sequence of $D$-negative flips relative to $g\circ f$, $k$ is an elementary divisorial contraction with exceptional divisor the transform $\wi{D}\subset\wi{X}$ of $D$, and $f_1$ is a contraction of fiber type with $\dim\ker (f_1)_*=2$. By \ref{no32} and Th.-Def.~\ref{long}, $k$ is of type  $(3,0)^{sm}$, $(3,0)^Q$, or $(3,1)^{sm}$; in particular $\w{X}_1$ has at most one isolated locally factorial and terminal singularity. Moreover $f_1$ is special, so that $Y_0$ has locally factorial, canonical singularities by Lemma \ref{lf}.
\end{pargtwo}
\begin{pargtwo}\label{induction}
Let us consider the factorization of $\psi$ as a sequence of $D$-negative flips relative to $g\circ f$:
$$\xymatrix{{\w{X}=Z_0}\ar@{-->}[r]^{\sigma_1}\ar[drr]_{g\circ f}&{\cdots}\ar@{-->}[r]
&{Z_{i-1}}\ar[d]_{\zeta_{i-1}}\ar@{-->}[r]^{\sigma_i}&{Z_{i}}\ar@{-->}[r]\ar[dl]_{\zeta_{i}}&{\cdots}\ar@{-->}[r]^{\sigma_n}
&
{Z_n=\wi{X}}\ar[dlll]^{f_1\circ k}\\
&&{Y_0}&&}$$
With a slight abuse of notation, we still denote by $D,E_1,E_2$ the transforms of these divisors in $Z_i$, for $i=0,\dotsc,n$.

We show by induction on $i=0,\dotsc,n$ that $\sigma_i$ is $K$-negative
and that $(E_1+E_2)\cdot\ell\leq 0$
 for every exceptional line $\ell\subset Z_i$.
For $i=0$, this holds by \ref{prel}.

Suppose that the statement is true for $i-1$. Let $R$ and $R'$ be the small extremal rays of $\NE(Z_{i-1})$ and $\NE(Z_i)$ respectively corresponding to the flip  $\sigma_i$. 
By the commutativity of the diagram above and by \ref{frankfurt}, we have $E_1+E_2+aD=\zeta_{i-1}^*(g(B))$, hence $(E_1+E_2+aD)\cdot R=0$, where $a>0$.
On the other hand
 $D\cdot R<0$, thus
 $(E_1+E_2)\cdot R>0$ and $(E_1+E_2)\cdot R'<0$.

 If $-K_{Z_{i-1}}\cdot R\leq 0$, then by \cite[Rem.~3.6(2)]{eff} there exists an exceptional line $\ell_0\subset Z_{i-1}$ such that $[\ell_0]\in R$, therefore $(E_1+E_2)\cdot\ell_0>0$, contradicting the induction assumption. Hence $-K_{Z_{i-1}}\cdot R> 0$ and $\sigma_i$ is $K$-negative.

Finally if $\ell\subset Z_i$ is an exceptional line, by \cite[Rem.~4.2]{eff} we have either $\ell\subset\dom\sigma_i^{-1}$, or $\ell\cap\dom\sigma_i^{-1}=\emptyset$. In the first case $\sigma_i^{-1}(\ell)$ is an exceptional line in $Z_{i-1}$, and we deduce that $(E_1+E_2)\cdot\ell\leq 0$. In the second case, we must have $[\ell]\in R'$ and hence $(E_1+E_2)\cdot\ell< 0$.
\end{pargtwo}
\begin{pargtwo}\label{flights}
 By \ref{induction}, $\psi$ factors as a sequence of $K$-negative flips, and Lemma \ref{basic1}$(c)$ yields that
 the indeterminacy locus of $\psi^{-1}$ is a disjoint union of exceptional lines
$\ell_1,\dotsc,\ell_s$.
\end{pargtwo}
\begin{pargtwo}\label{fiber}
Set $F_p:=f_1^{-1}(p)$. We show that  $\dim F_p=1$.

Note that $\w{X}$ and $\wi{X}$ are isomorphic outside the fibers of $g\circ f$ and $f_1\circ k$ over $p$, respectively. In $\w{X}$ we have  $(g\circ f)^{-1}(p)=D$, and the indeterminacy locus of $\psi$ must be contained in $D$. In 
$\wi{X}$ we have
 $(f_1\circ k)^{-1}(p)=k^{-1}(F_p)=\wi{D}\cup\overline{F}_p$, where $\overline{F}_p$ is the transform of the components of $F_p$ not contained in $k(\wi{D})$. On the other hand, by \ref{flights} we also
 have $k^{-1}(F_p)=\wi{D}\cup \ell_1\cup\cdots\cup\ell_s$.
This shows that $\overline{F}_p\subseteq \ell_1\cup\cdots\cup\ell_s$, in particular
$\dim\overline{F}_p\leq 1$, and since $\dim k(\wi{D})\leq 1$ (see \ref{frankfurt}), we conclude that $\dim F_p=1$.

 We have also shown that the transform in $\wi{X}$ of any irreducible  component of $F_p$ not contained in $k(\wi{D})$ must be one of the $\ell_i$'s.
\end{pargtwo}
\begin{pargtwo}\label{f1}
We show that $f_1$ is $K$-negative. Since $f$ is $K$-negative and $f_{|\w{X}\smallsetminus D}\cong (f_1)_{|\w{X}_1\smallsetminus F_p}$, we only have to check the fiber $F_p$. Let $\Gamma$ be an irreducible component of $F_p$.

If $\Gamma\not\subseteq k(\wi{D})$, then by \ref{fiber} we can assume that the transform of $\Gamma$ in $\wi{X}$ is $\ell_1$. 
Since $k^{-1}(F_p)$ is connected and $\ell_1,\dotsc,\ell_s$ are pairwise disjoint, we have 
$\wi{D}\cdot\ell_1>0$; notice also that $K_{\wi{X}}\cdot\ell_1=1$. Thus $-K_{\w{X}_1}\cdot\Gamma> 0$ because $k^*(-K_{\w{X}_1})=-K_{\wi{X}}+b\wi{D}$ with $b\in\{2,3\}$ (see \ref{frankfurt}).

If instead $\Gamma\subseteq k(\wi{D})$, then by \ref{frankfurt}
 $k$ must be of type $(3,1)^{sm}$ and $\Gamma=k(\wi{D})$. By \cite[Lemma 5.25]{blowup} there is a SQM $\ph_1\colon \w{X}_1\dasharrow X_1$ where $X_1$ is a Fano $4$-fold, and $\Gamma$ is contained in the open subset where $\ph_1$ is an isomorphism, so that $-K_{\w{X}_1}\cdot \Gamma=-K_{X_1}\cdot\ph_1(\Gamma)>0$.
\end{pargtwo}
\begin{pargtwo}
By \ref{frankfurt}, \ref{fiber}, and \ref{f1}, $\w{X}_1$ has isolated locally factorial and terminal singularities, $Y_0$ has  locally factorial canonical singularities, $f_1$ is $K$-negative, and $\dim F_p=1$.
Then \cite[Lemma 5.5]{ou} yields that
  $p$ is a smooth point of $Y_0$ (note that in \cite{ou} the contraction is supposed to be elementary, but this is used only to conclude that $Y_0$ is locally factorial, which here we already know). 
  
  In particular $p$ is a terminal singularity, hence $g$ is $K$-negative. The possibilities 
for $(G,-K_{\w{X}_1|G})$ are given in \cite[Th.~1.19]{AWaview};
moreover we know that $G$ is Gorenstein, and by adjunction that $-K_{G}\cdot C\geq 2$ for every curve $C\subset G$. Going through the list, it is easy to see that the possibilities for $G$ are $\pr^2$, $\pr^1\times\pr^1$, and the quadric cone. In the first two cases, $G\subset Y_{reg}$, and it follows from \cite[Cor.~3.4]{moriannals} that $G\cong\pr^2$ and  
$g$ is the blow-up of $p$. 

Suppose instead that $G$ is isomorphic to a quadric cone $Q$. Then the  normal bundle of $G$ has to be $\ol_Q(-1)$, and as in \cite[p.~164]{moriannals} and \cite[proof of Th.~5]{cutkosky} one sees that $\ma{I}_p\ol_Y=\ol_Y(-G)$ where $\ma{I}_p$ is the ideal sheaf of $p$ in $Y_0$, so that $g^{-1}(p)=G$ scheme-theoretically. Then $g$ factors through the blow-up of $p$, and being $g$ elementary, it must be the blow-up of $p$, which yields $G\cong\pr^2$ and hence a contradiction.
\qedhere
\end{pargtwo}
\end{proof}
\begin{parg}\label{nofibertype}
If $Y$ has an elementary rational contraction of fiber type $Y\dasharrow Z$, then $\rho_Z=\rho_X-3\geq 4$, in particular $Z$ is a surface. The composition $X\dasharrow Z$ is a rational contraction with $\rho_X-\rho_Z=3$, and we can apply Th.~\ref{surf}. If $(i)$ or $(ii)$ hold, we have the statement. If $(iii)$ holds, then $\rho_X\geq 13$ and $Z$ is a del Pezzo surface, so that $\rho_Z\leq 9$, which is impossible. Finally $(iv)$  cannot hold because $\rho_Z>1$.

 Therefore we can assume that $Y$ does not have elementary rational contractions of fiber type.
\end{parg}
\begin{parg}
Let $R$ be an extremal ray of $\NE(Y)$. By \ref{nofibertype} the associated contraction cannot be of fiber type, thus it is birational, either small of divisorial. By \ref{small} and   \ref{divisorial},  $-K_Y\cdot R\geq 0$. Since $Y$ is log Fano, $\NE(Y)$ is closed and polyhedral, and we conclude that $-K_Y$ is nef and
 $Y$ is a weak Fano variety (see \ref{prel}). 
\end{parg}
\begin{parg}\label{zero}
Let $Y\dasharrow \w{Y}$ be a SQM. Then the composition $X\dasharrow \w{Y}$ is again a special rational contraction with $\rho_X-\rho_{\w{Y}}=2$, so 
all the previous steps apply to $\w{Y}$ as well.
As in \cite[p.~622]{eff}, using \ref{small} and \ref{divisorial} one shows that 
if $E\subset Y$ is a fixed prime divisor, then $E$ can
contain at most finitely many
 curves of anticanonical degree zero.
\end{parg}
\begin{parg}
Let us consider all the contracting birational maps $Y\dasharrow Y_1$ with $\Q$-factorial target, and choose one with $\rho_{Y_1}$ minimal. 

Suppose that $\rho_{Y_1}\geq 3$.
By minimality, $Y_1$ has an elementary rational contraction of fiber type $Y_1\dasharrow Z$, and
$Z$ must be a surface with $\rho_Z=\rho_{Y_1}-1\geq 2$. 
The composition $X\dasharrow Z$ is a rational contraction, let  $F\subset X$ be  a general fiber. 
The general fiber of $Y\dasharrow Z$ is a smooth rational curve $\Gamma\subset Y$, and $\dim\N(F,X)\leq\dim\N(\Gamma,Y)+(\rho_X-\rho_Y)=3$.
Thus we get the statement by
Th.~\ref{surf}. 

Therefore we can assume that $\rho_{Y_1}\leq 2$.
\end{parg}
\begin{parg}
By \cite[Lemma 4.18]{blowup}, we can factor the map  $Y\dasharrow Y_1$ as $Y\dasharrow Y'\to Y_1$, where $Y\dasharrow Y'$ is a SQM, and $Y'\to Y_1$ is a sequence of elementary divisorial contractions. Now notice that the composition 
$X\dasharrow Y'$ is again a special rational contraction with $\rho_X-\rho_{Y'}=2$, so up to replacing $Y$ with $Y'$, we can assume that the map $a\colon Y\dasharrow Y_1$ is regular and is a sequence of $r:=\rho_Y-\rho_{Y_1}$ elementary divisorial contractions:
$$Y=W_0\stackrel{a_1}{\la}W_1\stackrel{a_2}{\la} W_2\la\cdots\la W_r=Y_1.$$
Let us show that the exceptional loci of these maps are all disjoint, so that $a$ is just the blow-up of $r$ distinct smooth points of $Y_1$.

We know by \ref{divisorial} that $a_1$ is the blow-up of a smooth point $w_1\in W_1$, and since $-K_Y$ is nef, it is easy to see that if $C\subset W_1$ is an irreducible curve containing $w_1$, then $-K_{W_1}\cdot C\geq 2$. 

Suppose that $\Exc(a_2)$ contains $w_1$. Then $a_2$ is $K$-negative, and $\Exc(a_2)$ cannot be covered by curves of anticanonical degree one. By \cite[Th.~1.19]{AWaview} this implies that $\Exc(a_2)\cong\pr^2$ and $(-K_{W_1})_{|\Exc(a_2)}\cong\ol_{\pr^2}(2)$. Then the transform of $\Exc(a_2)$ would be a fixed prime divisor covered by curves of anticanonical degree zero, which is impossible by \ref{zero}.
Proceeding in the same way, we conclude that the exceptional loci of the maps $a_i$ are all disjoint.

Now $Y_1$ is weak Fano with isolated locally factorial, canonical singularities, and we have $(-K_{Y_1})^3\leq 72$ by \cite{prok}. Therefore
$$0<(-K_Y)^3=(-K_{Y_1})^3-8r,$$ 
which yields $r\leq 8$ and $\rho_X=\rho_{Y_1}+r+2\leq 12$.\qedhere
\end{parg}
\end{proof}
Th.~\ref{main} is a straightforward consequence of Theorems \ref{surf} and \ref{3folds}.
\begin{proof}[Proof of Th.~\ref{rationalfibration}]
Let $X_0\subseteq X$ and $Y_0\subseteq Y$ be open subsets such that $f_0:=f_{_|X_0}\colon X_0\to Y_0$ is a projective morphism. Up to taking the Stein factorization, we can assume that $f_0$ is a contraction. 
Let $A\in\Pic(Y)$ be ample and consider $H:=f^*A\in\Pic(X)$. Then $H$ is a movable divisor, hence it yields a rational contraction $f'\colon X\dasharrow Y'$. It is easy to see that $f'_{|X_0}=f_0$, in particular $\dim Y'=3$. Then the statement follows from Th.~\ref{3folds}.
\end{proof}
\section{Fano $4$-folds to $\pr^1$}
\noindent Let $X$ be a Fano $4$-fold and $f\colon X\dasharrow\pr^1$ be a rational contraction; notice that $f$ is always special. In the following proposition we collect the information that we can give on $f$.
\begin{proposition}
Let $X$ be a smooth Fano $4$-fold and $f\colon X\dasharrow\pr^1$ be a rational contraction. Let $F_1,\dotsc,F_m$ be the reducible fibers of $f$.
Then one of the following holds:
\begin{enumerate}[$(i)$]
\item $\rho_X\leq 12$;
\item $X$ is a product of surfaces;
\item $\rho_X\leq m+10$, $f$ is not regular, and  every $F_i$ has two irreducible components, which are fixed divisors of type $(3,1)^{sm}$ or $(3,0)^Q$.
\end{enumerate}
\end{proposition}
\begin{proof}
We can assume that $\rho_X\geq 7$, so that $r_i=2$ for $i=1,\dotsc,m$ by 
Lemma \ref{r_i=2}.
By Lemma \ref{basic2} we can factor $f$ as $X\stackrel{\ph}{\dasharrow}X'\stackrel{f'}{\to}\pr^1$ where $\ph$ is a SQM, $X'$ is smooth, and  $f'$ is regular and $K$-negative. 

If some $F_i$ has a component of type $(3,0)^{sm}$, then we get $(i)$ by \cite[Th.~5.40]{blowup}.

If some $F_i$ has a component of type $(3,2)$, let $E\subset X'$ be its transform. Then $\N(E,X')\subseteq\ker(f')_*\subsetneq\N(X')$, so we get $(i)$ or $(ii)$ by Lemma \ref{prop32}.

We are left with the case where every component of every $F_i$ is of type $(3,1)^{sm}$ or $(3,0)^Q$.
The general fiber $F$ of $f'$ is a smooth Fano $3$-fold, so that $\rho_F\leq 10$ by Mori and Mukai's classification (see \cite[Coroll.~7.1.2]{fanoEMS}).
If $f$ is regular, then $\ph$ is an isomorphism, and $\rho_X\leq \rho_F+\delta_X$, so we get $(i)$ or $(ii)$  by Th.~\ref{buonconsiglio}.

If insteaf $f$ is not regular, then
 as in \cite[proof of Cor.~3.9]{eff} one shows that in fact $\rho_F\leq 9$. Therefore Cor.~\ref{boundrho} yields $\rho_X\leq m+10$, and we have $(iii)$.
\end{proof}
\providecommand{\noop}[1]{}
\providecommand{\bysame}{\leavevmode\hbox to3em{\hrulefill}\thinspace}
\providecommand{\MR}{\relax\ifhmode\unskip\space\fi MR }
\providecommand{\MRhref}[2]{%
  \href{http://www.ams.org/mathscinet-getitem?mr=#1}{#2}
}
\providecommand{\href}[2]{#2}

\end{document}